\documentclass[twoside,a4paper,12pt]{article}
\usepackage{amsfonts}
\usepackage{amsmath}
\usepackage{amsthm}
\usepackage{makeidx}

\usepackage[chatter]{rotating}

\numberwithin{equation}{section}

\PassOptionsToPackage{pdfauthor={Vladimir V. Kisil},%
  pdftitle={The Real and Complex
    Techniques in Harmonic Analysis from the Point of View of Covariant Transform},%
  pdfsubject={mathematics},%
  pdfkeywords={symmetries, wavelet, coherent state, covariant
    transform, reconstruction formula, the affine group, ax+b-group,
    square integrable representations, admissible vectors, Hardy
    space, fiducial operator, approximation of the identity, atom,
    nucleus, atomic decomposition, Cauchy integral, Poisson
    integral, Hardy--Littlewood maximal function, grand maximal
    function, vertical maximal function, non-tangential maximal
    function, intertwining operator, Cauchy-Riemann operator,
    Laplace operator, singular integral operator (SIO), Hilbert
    transform, boundary behaviour, Carleson measure, Littlewood--Paley theory}}{hyperref}

\usepackage[breaklinks=true,%
  bookmarks=true,%
  backref=page,%
  pagebackref=true
]{hyperref}

\usepackage[initials,nobysame]{amsrefs}\relax

\BibSpecAlias{incollection}{inproceedings}
\BibSpec{article}{%
    +{}  {\PrintAuthors}                {author}
    +{,} { \emph}                            {title}
    +{.} { }                            {part}
    +{:} { }                            {subtitle}
    +{.} { \PrintContributions}         {contribution}
    +{.} { \PrintPartials}              {partial}
    +{.} { }                       {journal}
    +{,} { }                            {volume}
    +{} { \parenthesize}                 {date}
    +{, No.}  { }               {number}
    +{, } {}                             {pages}
    +{,} { }                            {status}
    +{.} { \PrintTranslation}           {translation}
    +{.} { Reprinted in \PrintReprint}  {reprint}
    +{.} { }                            {note}
    +{.} {}                             {transition}
}

\BibSpec{partial}{%
    +{}  {}                             {part}
    +{:} { }                            {subtitle}
    +{.} { \PrintContributions}         {contribution}
    +{.} { \emph}                       {journal}
    +{,} { \textbf}                            {volume}
    +{}  { \parenthesize}               {number}
    +{:} {}                             {pages}
    +{,} { \PrintDateB}                 {date}
}

\BibSpec{book}{%
    +{}  {\PrintPrimary}                {transition}
    +{.} { \emph}                       {title}
    +{.} { }                            {part}
    +{:} { \emph}                       {subtitle}
    +{.} { }                            {series}
    +{,} { \voltext}                    {volume}
    +{.} { Edited by \PrintNameList}    {editor}
    +{.} { Translated by \PrintNameList}{translator}
    +{.} { \PrintContributions}         {contribution}
    +{.} { }                            {publisher}
    +{.} { }                            {organization}
    +{,} { }                            {address}
    +{,} { \PrintEdition}               {edition}
    +{,} { \PrintDateB}                 {date}
    +{.} { }                            {note}
    +{.} {}                             {transition}
    +{.} { \PrintTranslation}           {translation}
    +{.} { Reprinted in \PrintReprint}  {reprint}
    +{.} {}                             {transition}
}

\BibSpec{collection.article}{%
    +{}  {\PrintAuthors}                {author}
    +{.} { \emph}                            {title}
    +{.} { }                            {part}
    +{:} { }                            {subtitle}
    +{.} { \PrintContributions}         {contribution}
    +{.} { \PrintConference}            {conference}
    +{.} { \PrintBook}                  {book}
    +{.} { In \PrintEditorsA}              {editor}
    +{.} { }                         {booktitle}
    +{,} { pages~}                      {pages}
    +{,} { }                            {publisher}
    +{,} { }                            {organization}
    +{,} { }                            {address}
    +{,} { \PrintDateB}                 {date}
    +{.} { \PrintTranslation}           {translation}
    +{.} { Reprinted in \PrintReprint}  {reprint}
    +{.} { }                            {note}
    +{.} {}                             {transition}
}

 \BibSpec{inproceedings}{%
     +{}  {\PrintAuthors}                {author}
     +{.} { \emph}                            {title}
     +{.} { }                            {part}
     +{:} { }                            {subtitle}
     +{.} { \PrintContributions}         {contribution}
     +{.} { \PrintConference}            {conference}
     +{.} { \PrintBook}                  {book}
     +{.} { In \PrintEditorsA}              {editor}
     +{.} {  }                         {booktitle}
     +{,} { pages~}                      {pages}
     +{,} { }                            {publisher}
     +{,} { }                            {organization}
     +{,} { }                            {address}
     +{,} { \PrintDateB}                 {date}
     +{.} { \PrintTranslation}           {translation}
     +{.} { Reprinted in \PrintReprint}  {reprint}
     +{.} { }                            {note}
     +{.} {}                             {transition}
 }

\BibSpec{conference}{%
    +{}  {}                        {title}
    +{}  {\PrintConferenceDetails} {transition}
}

\BibSpec{innerbook}{%
    +{.} { \emph}                       {title}
    +{.} { }                            {part}
    +{:} { \emph}                       {subtitle}
    +{.} { }                            {series}
    +{,} { \voltext}                    {volume}
    +{.} { Edited by \PrintNameList}    {editor}
    +{.} { Translated by \PrintNameList}{translator}
    +{.} { \PrintContributions}         {contribution}
    +{.} { }                            {publisher}
    +{.} { }                            {organization}
    +{,} { }                            {address}
    +{,} { \PrintEdition}               {edition}
    +{,} { \PrintDateB}                 {date}
    +{.} { }                            {note}
    +{.} {}                             {transition}
}

\BibSpec{report}{%
    +{}  {\PrintPrimary}                {transition}
    +{.} { \emph}                       {title}
    +{.} { }                            {part}
    +{:} { \emph}                       {subtitle}
    +{.} { \PrintContributions}         {contribution}
    +{.} { Technical Report }           {number}
    +{,} { }                            {series}
    +{.} { }                            {organization}
    +{,} { }                            {address}
    +{,} { \PrintDateB}                 {date}
    +{.} { \PrintTranslation}           {translation}
    +{.} { Reprinted in \PrintReprint}  {reprint}
    +{.} { }                            {note}
    +{.} {}                             {transition}
}

\BibSpec{thesis}{%
    +{}  {\PrintAuthors}                {author}
    +{,} { \emph}                       {title}
    +{:} { \emph}                       {subtitle}
    +{.} { \PrintThesisType}            {type}
    +{.} { }                            {organization}
    +{,} { }                            {address}
    +{,} { \PrintDateB}                 {date}
    +{.} { \PrintTranslation}           {translation}
    +{.} { Reprinted in \PrintReprint}  {reprint}
    +{.} { }                            {note}
    +{.} {}                             {transition}
}

\newtheorem{thm}{Theorem}[section]
\newtheorem{prop}[thm]{Proposition}

\newtheorem{cor}[thm]{Corollary}
\newtheorem{lem}[thm]{Lemma}
\theoremstyle{definition}
\newtheorem{defn}[thm]{Definition}
\newtheorem{example}[thm]{Example}
\theoremstyle{remark}
\newtheorem{rem}[thm]{Remark}

\providecommand{\comment}[1]{}
\providecommand{\linv}[2][\relax]{\mathfrak{L}^{#2}_{#1}}
\providecommand{\norm}[2][\relax]{\left\|#2\right\|\ifx#1\relax\else_{#1}\fi}
\providecommand{\modulus}[2][\relax]{\left| #2 \right|\ifx#1\relax\else_{#1}\fi}

\providecommand{\uir}[3][0]{\ifcase #1{\rho^{#2}_{#3}}%
\or {\breve{\rho}^{#2}_{#3}}%
\or {\tilde{\rho}^{#2}_{#3}}%
\or {\hat{\rho}^{#2}_{#3}}\fi}

\providecommand{\Space}[3][]{\ensuremath{\mathbb{#2}^{#3}_{#1}{}}}
\providecommand{\such}{\,\mid\,}

  \providecommand{\FSpace}[3][]{\ensuremath{\ifx#2l \ell_{#3}^{#1}{}\else
  #2_{#3}^{#1}{}\fi}} 
\providecommand{\oper}[1]{\mathcal{#1}}
\providecommand{\rmi}{\mathrm{i}}
\providecommand{\rme}{\mathrm{e}}
\providecommand{\scalar}[3][\relax]{\left\langle #2,#3 
        \right\rangle\ifx#1\relax\else_{#1}\fi}
\providecommand{\algebra}[1]{\ensuremath{\mathfrak{#1}}}
  \providecommand{\Zbl}[1]{Zbl\href{http://www.emis.de:80/cgi-bin/zmen/ZMATH/en/zmathf.html?first=1&maxdocs=3&type=html&an=#1&format=complete}{#1}}
\providecommand{\myeprint}[2]{E-print: \href{#1}{\texttt{#2}}}
\providecommand{\SL}[1][2]{\ensuremath{\mathrm{SL}_{#1}(\Space{R}{})}}

\let\oldexample=\example
\let\oldendexample=\endexample
\renewenvironment{example}{\oldexample}{\(\diamond\)\oldendexample}
\makeindex

\begin{document}
\thispagestyle{plain}
 \sloppy

 \begin{center}
\textbf{  THE REAL AND COMPLEX
  TECHNIQUES\\ IN HARMONIC ANALYSIS\\FROM THE POINT OF VIEW OF COVARIANT
  TRANSFORM}
 \end{center}

\vskip 0.3cm

\centerline{\bf {\href{http://www.maths.leeds.ac.uk/~kisilv/}{Vladimir V. Kisil}} }

\markboth{\hfill{\footnotesize\rm   V.V.~Kisil  }\hfill}
{\hfill{\footnotesize\sl  The real and complex
  techniques in harmonic analysis}\hfill}

\vskip 0.3cm

\vskip 0.7 cm





\centerline{\emph{Dedicated to Prof. Sergei V. Rogosin on the occasion of his 60th birthday}}

\vskip 0.2cm

\noindent {\bf Key words:}  wavelet, coherent state, covariant transform, reconstruction
  formula, the affine group, \(ax+b\)-group, square integrable
  representations, admissible vectors, Hardy space, fiducial operator,
  approximation of the identity, atom, nucleus, atomic decomposition,
  Cauchy integral, Poisson integral, Hardy--Littlewood maximal
  function, grand maximal function, vertical maximal function,
  non-tangential maximal function, intertwining operator,
  Cauchy-Riemann operator, Laplace operator, singular integral
  operator, SIO, Hilbert transform, boundary behaviour, Carleson
  measure, Littlewood--Paley theory

\vskip 0.2cm


\vskip 0.2cm

\noindent {\bf Abstract.}
  This paper reviews  complex and real techniques in harmonic
  analysis. We describe the common source of both approaches rooted in
  the covariant transform generated by the affine group.

\tableofcontents

\section{\large Introduction}
\label{sec:introduction}

There are two main approaches in harmonic analysis on the real line.
The real variables technique uses various maximal functions, dyadic
cubes and, occasionally, the Poisson integral~\cite{Stein93}.  The
complex variable technique is based on the Cauchy integral and fine
properties of analytic functions~\cites{Nikolski02a,Nikolski02b}.

Both methods seem to have clear advantages. The real variable
technique:
\begin{enumerate}
\item does not require an introduction of the imaginary unit for a
  study of real-valued harmonic functions of a real variable (Occam's
  Razor: among competing hypotheses, the one with the fewest
  assumptions should be selected); 
\item allows a straightforward generalization to several real variables. 
\end{enumerate}
By contrast, access to the beauty and power of analytic functions
(e.g., M\"obius transformations, factorisation of zeroes,
etc.~\cite{Koosis98a}) is the main reason to use the complex variable
technique. A posteriori, a multidimensional analytic version was also
discovered~\cite{McIntosh95a}, it is based on the monogenic
Clifford-valued functions~\cite{BraDelSom82}.

Therefore, propensity for either techniques becomes a personal choice
of a researcher. Some of them prefer the real variable method,
explicitly cleaning out any reference to analytic or harmonic
functions~\cite{Stein93}*{Ch.~III, p.~88}. Others,
e.g.~\cite{Krantz09a,CoifmanJonesSemmes89}, happily combine the both
techniques. However, the reasons for switching between two methds at
particular places may look mysterious.

The purpose of the present paper is to revise the origins of the real
and complex variable techniques. Thereafter, we describe the common
group-theoretical root of both. Such a unification deepens our
understanding of both methods and illuminates their interaction.

\begin{rem}
  In this paper, we consider only examples which are supported by the
  affine group \(\mathrm{Aff}\) of the real line. In the essence,
  \(\mathrm{Aff}\) is the semidirect product of the group of dilations
  acting on the group of translations. Thus, our consideration can be
  generalized to the semidirect product of dilations and homogeneous
  (nilpotent) Lie groups, cf.~\cites{FollStein82,Kisil12b}. Other
  important extensions are the group \(\SL\)\index{$\SL$}%
  \index{group!$\SL$} and associated hypercomplex algebras, see
  Rems.~\ref{re:hardy-SL2},~\ref{re:hypercomplex-numbers}
  and~\cites{Kisil05a,Kisil12a,Kisil11c}. However, we do not aim here
  to a high level of generality, it can be developed in subsequent
  works once the fundamental issues are sufficiently clarified.
\end{rem}

\section{\large Two approaches to harmonic analysis}
\label{sec:preliminaries}

As a starting point of our discussion, we provide a schematic outline
of complex and real variables techniques in the one-dimensional harmonic
analysis.  The application of complex analysis may be
summarised in the following sequence of principal steps:
\begin{description}
\item[Integral transforms.] For a function 
  \(f\in\FSpace{L}{p}(\Space{R}{})\), we apply the Cauchy%
  \index{integral!Cauchy}%
  \index{Cauchy!integral} or Poisson%
  \index{Poisson kernel}%
  \index{kernel!Poisson} integral transforms:
  \begin{align}
    \label{eq:cauchy-int}
    [\oper{C} f] (x+\rmi y)& = \frac{1}{2\pi \rmi} \int_{\Space{R}{}}
    \frac{f(t)}{t-(x+\rmi y)}\, dt\,,\\
    \label{eq:poisson-int}
    [\oper{P}f](x, y) & =\frac{1}{\pi}\int_{\Space{R}{}} \frac {y}{(t-x)^2 + y^2} \,
    f(t)\, dt \,.
  \end{align}
  An equivalent transformation on the unit circle replaces the
  Fourier series \(\sum_k c_k e^{\rmi kt}\) by the Taylor series 
  \(\sum_{k=0}^\infty c_k z^k\) in the complex variable \(z=r e^{\rmi t}\), \(0\leq
  r<1\). It is used for the Abel summation%
  \index{Abel summation}%
  \index{summation!Abel} of trigonometric series~\cite{Zygmund02}*{\S~III.6}.  
\item[Domains.] Above
  integrals~\eqref{eq:cauchy-int}--\eqref{eq:poisson-int} map the
  domain of functions from the real line to the upper half-plane,
  which can be conveniently identified with the set of complex numbers
  having a positive imaginary part. The larger domain allows us to
  inspect functions in greater details.
\item[Differential operators.] The image of
  integrals~\eqref{eq:cauchy-int} and~\eqref{eq:poisson-int} consists
  of functions,belonging to the kernel of the Cauchy--Riemann operator
  \(\partial_{\bar{z}}\) and Laplace operator  \(\Delta\)
  respectively, i.e.:
  \begin{equation}
    \label{eq:CR-Laplace}
    \partial_{\bar{z}}=\frac{\partial}{\partial x}+\rmi
    \frac{\partial}{\partial y}\,,\qquad
    \Delta=\frac{\partial^2}{\partial x^2}+
    \frac{\partial^2}{\partial y^2}\,.
  \end{equation}
  Such functions have numerous nice properties in the upper
  half-plane, e.g. they are infinitely differentiable, which make
  their study interesting and fruitful.
\item[Boundary values and SIO.] To describe properties of the initial
  function \(f\) on the real line we consider the boundary values of
  \([\oper{C} f] (x+\rmi y)\) or \([\oper{P}f](x, y)\), i.e. their
  limits as \(y\rightarrow 0\) in some sense. The Sokhotsky--Plemelj
  formula%
  \index{Sokhotsky--Plemelj formula}%
  \index{formula!Sokhotsky--Plemelj} provides the boundary value of
  the Cauchy integral~\cite{MityushevRogosin00a}*{(2.6.6)}:
  \begin{equation}
    \label{eq:sokhotsky-plemelj}
    [\oper{C}f](x,0)=\frac{1}{2} f(x)+\frac{1}{2\pi\rmi}
    \int_{\Space{R}{}} \frac{f(t)}{t-x}\,dt.
  \end{equation}
  The last term is a singular integral operator%
  \index{singular!integral operator}%
  \index{operator!integral!singular} defined through the principal
  value%
  \index{principal!value}%
  \index{value!principal} in the Cauchy sense:
  \begin{equation}
    \label{eq:pv-cauchy}
    \frac{1}{2\pi\rmi}
    \int_{\Space{R}{}}
    \frac{f(t)}{t-x}\,dt=\lim_{\varepsilon\rightarrow 0}
    \frac{1}{2\pi\rmi}
    \int\limits_{-\infty}^{x-\varepsilon}+
    \int\limits_{x+\varepsilon}^{\infty}
    \frac{f(t)}{t-x}\,dt\, .
  \end{equation}
  For the Abel summation the boundary values are replaced by the limit
  as \(r\rightarrow 1^-\) in the series  \(\sum_{k=0}^\infty c_k (r e^{\rmi t})^k\).
\item[Hardy space.] Sokhotsky--Plemelj
  formula~\eqref{eq:sokhotsky-plemelj} shows, that the boundary value
  \([\oper{C}f](x,0)\) may be different from \(f(x)\). The vector
  space of functions \(f(x)\) such that \([\oper{C}f](x,0)=f(x)\) is
  called the Hardy space%
  \index{space!Hardy}%
  \index{Hardy!space} on the real line~\cite{Nikolski02a}*{A.6.3}.
 \end{description}
Summing up this scheme: we replace a function (distribution) on the
real line by a nicer (analytic or harmonic) function on a larger
domain---the upper half-plane. Then, we trace down properties of
the extensions to its boundary values and, eventually, to the initial
function. 

The real variable approach does not have a clearly designated path in
the above sense. Rather, it looks like a collection of interrelated
tools, which are efficient for various purposes. To highlight
similarity and differences between real and complex analysis, we line
up the elements of the real variable technique in the following way:
\begin{description}
\item[Hardy--Littlewood maximal function]%
  \index{Hardy--Littlewood!maximal functions}%
  \index{maximal functions!Hardy--Littlewood} is, probably, the most
  important component~\citelist{\cite{Koosis98a}*{\S~VIII.B.1}
    \cite{Stein93}*{Ch.~2} \cite{Garnett07a}*{\S~I.4} \cite{Burenkov12a}} of this technique. The maximal function
  \(f^M\) is defined on the real line by the identity:
  \begin{equation}
    \label{eq:hardy-littlewood}
    f^M(t)= \sup_{a>0} \left\{\frac{1}{2a} \int\limits^{t+a}_{t-a}
      \modulus{f\left(x\right)}\,dx\right\}.
  \end{equation}
\item[Domain] is not apparently changed, the maximal function \(f^M\)
  is again defined on the real line. However, an efficient treatment
  of the maximal functions requires consideration of tents%
  \index{tent}~\cite{Stein93}*{\S~II.2}, which are parametrised by their
  vertices, i.e.  points \((a,b)\), \(a>0\), of the upper
  half-plane. In other words, we repeatedly need values of all
  integrals \(\frac{1}{2a} \int\limits^{t+a}_{t-a}
  \modulus{f\left(x\right)}\,dx\), rather than the single value of the
  supremum
  over \(a\).
\item[Littlewood--Paley theory]%
  \index{Littlewood--Paley theory}%
  \index{theory!Littlewood--Paley}~\cite{CoifmanJonesSemmes89}*{\S~3}
  and associated dyadic squares\index{dyadic!squares}%
  \index{squares!dyadic}
  technique~\citelist{\cite{Garnett07a}*{Ch.~VII, Thm.~1.1}
    \cite{Stein93}*{\S~IV.3}} as well as stopping time
  argument~\cite{Garnett07a}*{Ch.~VI, Lem.~2.2} are based on bisection
  of a function's domain into two equal parts.
\item[SIO] is a natural class of bounded linear operators in
  \(\FSpace{L}{p}(\Space{R}{})\). Moreover, maximal operator \(M:
  f\rightarrow f^M\)~\eqref{eq:hardy-littlewood} and
  singular integrals are intimately related~\cite{Stein93}*{Ch.~I}.
\item[Hardy space] can be defined in several equivalent ways from
  previous notions. For example, it is the class of such functions
  that their image under maximal
  operator~\eqref{eq:hardy-littlewood} or singular
  integral~\eqref{eq:pv-cauchy} belongs to
  \(\FSpace{L}{p}(\Space{R}{})\)~\cite{Stein93}*{Ch.~III}.
\end{description}
The following discussion will line up real variable objects along the
same axis as complex variables. We will summarize this in
Table~\ref{tab:corr-betw-diff}.

\section{\large Affine group and its representations}
\label{sec:prel-affine-group}

It is hard to present harmonic analysis and wavelets without touching
the affine group one way or another. Unfortunately, many sources only
mention the group and do not use it explicitly. On the other hand, it
is equally difficult to speak about the affine group without a
reference to results in harmonic analysis: two theories are intimately
intertwined. In this section we collect fundamentals of the affine
group and its representations, which are not yet a standard background
of an analyst.

Let \(G=\mathrm{Aff}\) be the \(ax+b\)%
\index{$ax+b$ group}%
\index{group!$ax+b$} (or the \emph{affine}%
\index{affine!group|see{$ax+b$ group}}%
\index{group!affine|see{$ax+b$ group}})
group~\cite{AliAntGaz00}*{\S~8.2}, which is represented (as a
topological set) by the upper half-plane \(\{(a,b) \such
a\in\Space[+]{R}{}, \ b\in\Space{R}{}\}\). The group law is:
\begin{equation}
  \label{eq:ax+b-group-law}
  (a,b)\cdot (a',b')=(aa',ab'+b)
  .
\end{equation}
As any other group, \(\mathrm{Aff}\) has the \emph{left regular
  representation}%
\index{$ax+b$ group!left regular representation}%
\index{left regular representation!$ax+b$ group}%
\index{representation!left regular!$ax+b$ group}
by shifts on functions \(\mathrm{Aff}\rightarrow \Space{C}{}\):
\begin{equation}
  \label{eq:ax+b-left-regular}
  \Lambda(a,b): f(a',b') \mapsto f_{(a,b)}(a',b')=f\left(\frac{a'}{a},\frac{b'-b}{a}\right).
\end{equation}
A left invariant measure on \(\mathrm{Aff}\) is \(dg=a^{-2}\,da\,db\),
\(g=(a,b)\)%
\index{$ax+b$ group!invariant measure}%
\index{group!$ax+b$!invariant measure}%
\index{invariant!measure}%
\index{measure!invariant}. By the definition, the left regular
representation~\eqref{eq:ax+b-left-regular} acts by unitary operators
on \(\FSpace{L}{2}(\mathrm{Aff},dg)\). The group is not unimodular and
a right invariant measure is \(a^{-1}\,da\,db\).

There are two important subgroups of the \(ax+b\) group:
\begin{equation}
  \label{eq:subgroups-A-N}
  A=\{(a,0)\in \mathrm{Aff} \such a\in\Space[+]{R}{} \} \quad 
  \text{ and } \quad
  N=\{(1,b)\in \mathrm{Aff} \such b\in\Space{R}{} \}.
\end{equation}

An isometric representation of \(\mathrm{Aff}\) on
\(\FSpace{L}{p}(\Space{R}{})\) is given by the formula:
\begin{equation}
  \label{eq:ax+b-repr-quasi-reg}
  [\uir{}{p}(a,b)\, f](x)= a^{-\frac{1}{p}}\,f\left(\frac{x-b}{a}\right).
\end{equation}
Here, we identify the real line with the subgroup \(N\) or, even more
accurately, with the homogeneous space
\(\mathrm{Aff}/N\)~\cite{ElmabrokHutnik12a}*{\S~2}. This
representation is known as \emph{quasi-regular}%
\index{$ax+b$ group!representation!quasi-regular}%
\index{representation!$ax+b$ group!quasi-regular}%
\index{quasi-regular!representation of $ax+b$ group} for its
similarity with~\eqref{eq:ax+b-left-regular}. The action of the
subgroup \(N\) in~\eqref{eq:ax+b-repr-quasi-reg} reduces to shifts,
the subgroup \(A\) acts by dilations.

\begin{rem}
  The \(ax+b\) group definitely escapes Occam's Razor in
  harmonic analysis, cf. the arguments against the imaginary unit in
  the Introduction. Indeed, shifts are required to define convolutions
  on \(\Space{R}{n}\), and an \emph{approximation of the identity}%
  \index{approximation!identity, of the}%
  \index{identity!approximation of the}~\cite{Stein93}*{\S~I.6.1} is a
  convolution with the dilated kernel. The same scaled convolutions
  define the fundamental \emph{maximal
  functions}%
  \index{maximal functions}, see~\cite{Stein93}*{\S~III.1.2}
  cf. Example~\ref{ex:hardy-littlewood} below. Thus, we can avoid 
  usage of the upper half-plane \(\Space[+]{C}{}\), but the same set will
  anyway re-invent itself in the form of the \(ax+b\) group.
\end{rem}

The representation \eqref{eq:ax+b-repr-quasi-reg} in
\(\FSpace{L}{2}(\Space{R}{})\) is reducible and the space can be split
into irreducible subspaces. Following the philosophy presented in the
Introduction to the paper~\cite{Kisil09e}*{\S~1} we give the following
\begin{defn}
  \label{de:hardy-space}
  For a representation \(\uir{}{}\) of a group \(G\) in a space \(V\),
  a \emph{generalized Hardy space}%
  \index{space!Hardy!generalized}%
  \index{Hardy!space!generalized} \(\FSpace{H}{}\) is an
  \(\uir{}{}\)-irreducible%
  \index{representation!irreducible}%
  \index{irreducible!representation} (or \(\uir{}{}\)-primary%
  \index{representation!primary}%
  \index{primary!representation}, as discussed in
  Section~\ref{sec:comp-covar-transfr}) subspace of \(V\).
\end{defn}
\begin{example}
  Let \(G=\mathrm{Aff}\) and the representation \(\uir{}{p}\) be
  defined in \(V=\FSpace{L}{p}(\Space{R}{})\)
  by~\eqref{eq:ax+b-repr-quasi-reg}. Then the classical Hardy spaces
  \(\FSpace{H}{p}(\Space{R}{})\)%
  \index{space!Hardy}%
  \index{Hardy!space} are \(\uir{}{p}\)-irreducible, thus are covered
  by the above definition.
\end{example}
Some ambiguity in picking the Hardy space out of all (well, two, as we
will see below) irreducible components is resolved by the traditional
preference.
\begin{rem}
  We have defined the Hardy space completely in terms of
  representation theory of \(ax+b\) group. The traditional
  descriptions, via the Fourier transform or analytic
  extensions, will be corollaries  in our approach, see
  Prop.~\ref{pr:fourier-decompos}  and Example~\ref{ex:cauchy-riemann-wav}.
\end{rem}
\begin{rem}
  \label{re:hardy-SL2}
  It is an interesting and important observation, that the Hardy space
  in \(\FSpace{L}{p}(\Space{R}{})\) is invariant under the action of a
  larger group \(\SL\)\index{$\SL$}%
  \index{group!$\SL$}, the group of \(2\times 2\) matrices with real
  entries and determinant equal to \(1\), the group operation coincides
  with the multiplication of matrices. The \(ax+b\) group is
  isomorphic to the subgroup of the upper-triangular matrices in
  \(\SL\). The group \(\SL\) has an isometric
  representation in \(\FSpace{L}{p}(\Space{R}{})\):
  \begin{equation}
    \label{eq:sl2r-action-rl}
    \begin{pmatrix}
      a&b\\c&d
    \end{pmatrix}:\ f(x)\ \mapsto\
    \frac{1}{\modulus{a-cx}^{\frac{2}{p}}}\,f\left(\frac{dx-b}{a-cx}\right), 
  \end{equation}
  which produces quasi-regular
  representation~\eqref{eq:ax+b-repr-quasi-reg} by the restriction to
  upper-triangular matrices. The Hardy space
  \(\FSpace{H}{p}(\Space{R}{})\) is invariant under the above action
  as well. Thus, \(\SL\) produces a refined version in comparison with
  the harmonic analysis of the \(ax+b\) group considered in this
  paper. Moreover, as representations of the \(ax+b\) group are
  connected with complex numbers, the structure of \(\SL\) links
  all three types of hypercomplex numbers%
  \index{hypercomplex numbers}%
  \index{numbers!hypercomplex}~\citelist{\cite{Kisil05a}
    \cite{Kisil12a}*{\S~3.3.4} \cite{Kisil11c}*{\S~3}}, see also
  Rem.~\ref{re:hypercomplex-numbers}. 
\end{rem}

To clarify a decomposition of \(\FSpace{L}{p}(\Space{R}{})\) into
irreducible subspaces of 
representation~\eqref{eq:ax+b-repr-quasi-reg} we need another
realization of this representation. It is called \emph{co-adjoint}%
\index{$ax+b$ group!representation!co-adjoint}%
\index{representation!$ax+b$ group!co-adjoint}%
\index{co-adjoint!representation of $ax+b$ group} and is related to the
\emph{orbit method}%
\index{orbit!method}%
\index{method!orbit} of Kirillov~\citelist{
  \cite{Kirillov04a}*{\S~4.1.4} \cite{Folland95}*{\S~6.7.1}}. Again,
this isometric representation can be defined on
\(\FSpace{L}{p}(\Space{R}{})\) by the formula:
\begin{equation}
  \label{eq:ax+b-repr-co-adj}
  [\uir[3]{}{p}(a,b)\, f](\lambda )= a^{\frac{1}{p}}\,\rme^{-2\pi\rmi b
    \lambda } f(a\lambda ).
\end{equation}
Since \(a>0\), there is an obvious decomposition into invariant subspaces of \(\uir[3]{}{p}\):
\begin{equation}
  \label{eq:co-adjoint-inv-decomp}
  \FSpace{L}{p}(\Space{R}{})=\FSpace{L}{p}(-\infty,0)\oplus\FSpace{L}{p}(0,\infty).
\end{equation}
It is possible to demonstrate, that these components are irreducible.
This decomposition has a spatial nature, i.e., the subspaces have
disjoint supports. Each half-line can be identified with the subgroup
\(A\) or with the homogeneous space \(\mathrm{Aff}/N\).

The restrictions \(\uir[3]{+}{p}\) and \(\uir[3]{-}{p}\) of the
co-adjoint representation \(\uir[3]{}{p}\) to invariant
subspaces~\eqref{eq:co-adjoint-inv-decomp} for \(p=2\) are not unitary equivalent.
Any irreducible unitary representation of \(\mathrm{Aff}\) is unitary
equivalent either to \(\uir[3]{+}{2}\) or \(\uir[3]{-}{2}\). Although
there is no intertwining operator between \(\uir[3]{+}{p}\) and
\(\uir[3]{-}{p}\), the
map:
\begin{equation}
  \label{eq:J-map}
  {J}:\  \FSpace{L}{p}(\Space{R}{}) \rightarrow
  \FSpace{L}{p}(\Space{R}{}):\  {f}(\lambda )\mapsto {f}(-\lambda ),
\end{equation}
has the property
\begin{equation}
  \label{eq:J-swaps-reps}
  \uir[3]{-}{p}(a,-b)\circ{J}={J}\circ\uir[3]{+}{p}(a,b) 
\end{equation}
which corresponds to the
outer automorphism \((a,b)\mapsto (a,-b)\) of \(\mathrm{Aff}\).

As was already mentioned, for the Hilbert space \(\FSpace
{L}{2}(\Space{R}{})\), representations~\eqref{eq:ax+b-repr-quasi-reg}
and \eqref{eq:ax+b-repr-co-adj} are unitary equivalent, i.e., there is
a unitary intertwining operator between them. We may guess its
nature as follows. The eigenfunctions of the operators \(\uir{}{2}(1,b)\) are
\(\rme^{2\pi\rmi \omega  x}\)  and the eigenfunctions of
\(\uir[3]{}{2}(1,b)\) are \(\delta(\lambda -\omega )\). Both sets form
``continuous bases'' of \(\FSpace{L}{2}(\Space{R}{})\) and the
unitary operator which maps one to another is the Fourier transform:
\begin{equation}
  \label{eq:fourier}
  \oper{F}: f(x) \mapsto \hat{f}(\lambda)=\int_{\Space{R}{}} 
    e^{-2\pi\rmi \lambda x}\, f(x)\, dx.
\end{equation}
Although, the above arguments were informal, the intertwining property
\(\oper{F}\uir{}{2}(a,b)=\uir[3]{}{2}(1,b)\oper{F}\) can be directly
verified by the appropriate change of variables in the Fourier
transform. Thus, cf.~\cite{Nikolski02a}*{Lem.~A.6.2.2}:
\begin{prop}
  \label{pr:fourier-decompos}
  The Fourier transform maps  irreducible invariant subspaces
  \(\FSpace{H}{2}\) and \(\FSpace[\perp]{H}{2}\)
  of~\eqref{eq:ax+b-repr-quasi-reg} to irreducible invariant
  subspaces \(\FSpace{L}{2}(0,\infty)=\oper{F}(\FSpace{H}{2})\) and
  \(\FSpace{L}{2}(-\infty,0)=\oper{F}(\FSpace[\perp]{H}{2})\) of 
  co-adjoint representation~\eqref{eq:ax+b-repr-co-adj}. In
  particular, \(\FSpace {L}{2}(\Space{R}{})=\FSpace{H}{2}\oplus
  \FSpace[\perp]{H}{2}\). 
\end{prop}
Reflection \(J\)~\eqref{eq:J-map} anticommutes  with the Fourier
transform: \(\oper{F} J=-J \oper{F}\). Thus,  \(J\) also interchange
the irreducible components \(\uir{+}{p}\) and \(\uir{-}{p}\) of 
quasi-regular representation~\eqref{eq:ax+b-repr-quasi-reg} according
to~\eqref{eq:J-swaps-reps}.


Summing up, the unique r\^ole of the Fourier transform in harmonic
analysis is based on the following facts from the representation
theory. The Fourier transform
\begin{itemize}
\item  intertwines shifts in quasi-regular
  representation~\eqref{eq:ax+b-repr-quasi-reg} to operators of
  multiplication in co-adjoint
  representation~\eqref{eq:ax+b-repr-co-adj};
\item intertwines dilations
  in~\eqref{eq:ax+b-repr-quasi-reg} to dilations
  in~\eqref{eq:ax+b-repr-co-adj};
\item maps the decomposition
  \(\FSpace{L}{2}(\Space{R}{})=\FSpace{H}{2}\oplus
  \FSpace[\perp]{H}{2}\) into spatially separated spaces with
  disjoint supports;
\item anticommutes with \(J\), which interchanges \(\uir{+}{2}\) and
  \(\uir{-}{2}\).
\end{itemize}
Armed with this knowledge we are ready to proceed to harmonic analysis.

\section{\large Covariant transform}
\label{sec:covariant-transform}

We make an extension of the wavelet\index{wavelet} construction
defined in terms of group representations. See~\cite{Kirillov76} for a
background in the representation theory, however, the only treated
case in this paper is the \(ax+b\) group.
\begin{defn}\cites{Kisil09d,Kisil11c}
  \label{de:covariant-transform}
  Let \(\uir{}{}\) be a representation of
  a group \(G\) in a space \(V\) and \(F\) be an operator acting from \(V\) to a space
  \(U\). We define a \emph{covariant transform}%
  \index{covariant!transform}%
  \index{transform!covariant}
  \(\oper{W}_F^{\uir{}{}}\) acting from \(V\) to the space \(\FSpace{L}{}(G,U)\) of
  \(U\)-valued functions on \(G\) by the formula:
  \begin{equation}
    \label{eq:coheret-transf-gen}
    \oper{W}_F^{\uir{}{}}: v\mapsto \hat{v}(g) = F(\uir{}{}(g^{-1}) v), \qquad
    v\in V,\ g\in G.
  \end{equation}
  The operator \(F\) will be called a \emph{fiducial operator}%
  \index{fiducial operator}%
  \index{operator!fiducial} in this context (cf. the fiducial vector
  in~\cite{KlaSkag85}).
\end{defn}
We may drop the sup/subscripts from \(\oper{W}_F^{\uir{}{}}\) if the
functional \(F\) and/or the representation \(\uir{}{}\)
are clear from the context.
\begin{rem}
  We do not require that the fiducial operator \(F\) be linear.
  Sometimes the positive homogeneity, i.e.\ \(F(t v)=tF(v)\) for
  \(t>0\), alone can be already sufficient, see
  Example~\ref{ex:maximal-function}.
\end{rem}
\begin{rem}
  \label{re:range-dim} 
  It looks like the usefulness of the covariant transform is in the
  reverse proportion to the dimension of the space \(U\). The
  covariant transform encodes properties of \(v\) in a function
  \(\oper{W}_F^{\uir{}{}}v\) on \(G\), which is a scalar-valued function if
  \(\dim U=1\). However, such a simplicity is not always possible.
  Moreover, the paper~\cite{Kisil12b} gives an important example of a
  covariant transform which provides a simplification even in the case
   \(\dim U =\dim V\).
\end{rem}

We start the list of examples with the classical case of the group-theoretical
wavelet transform.
\begin{example}\cites{Perelomov86,FeichGroech89a,Kisil98a,AliAntGaz00,%
 KlaSkag85,FeichGroech89a} 
  \label{ex:wavelet}
  Let \(V\) be a Hilbert space with an inner product
  \(\scalar{\cdot}{\cdot}\) and \(\uir{}{}\) be a unitary
  representation of a group \(G\) in the space \(V\). Let \(F: V
  \rightarrow \Space{C}{}\) be the functional \(v\mapsto
  \scalar{v}{v_0}\) defined by a vector \(v_0\in V\). The vector
  \(v_0\) is often called the \emph{mother wavelet}%
  \index{mother wavelet}%
  \index{wavelet!mother} in areas related
  to signal processing, the \emph{vacuum state}%
  \index{vacuum state|see{mother wavelet}}%
  \index{state!vacuum|see{mother wavelet}} in the quantum
  framework, etc.

  In this set-up, transformation~\eqref{eq:coheret-transf-gen} is
  the well-known expression for a \emph{wavelet transform}%
  \index{wavelet!transform}%
  \index{transform!wavelet}~\cite{AliAntGaz00}*{(7.48)} (or
  \emph{representation coefficients}%
  \index{representation!coefficients|see{wavelet transform}}):
  \begin{equation}
    \label{eq:wavelet-transf}
    \oper{W}: v\mapsto \tilde {v}(g) = \scalar{\uir{}{}(g^{-1})v}{v_0}  =
    \scalar{ v}{\uir{}{}(g)v_0}, \qquad
    v\in V,\ g\in G.
  \end{equation}
  The family of the vectors \(v_g=\uir{}{}(g)v_0\) is called
  \emph{wavelets}\index{wavelet} or \emph{coherent states}. The image
  of \eqref{eq:wavelet-transf} consists of
  scalar valued functions on \(G\).
\end{example}

This scheme is typically carried out for a square integrable
representation \(\uir{}{}\)%
\index{square integrable!representation}%
\index{representation!square integrable} with \(v_0\) being an
admissible vector%
\index{admissible wavelet}%
\index{wavelet!admissible}~\cites{Perelomov86,FeichGroech89a,%
AliAntGaz00,Fuhr05a,ChristensenOlafsson09a,DufloMoore},
i.e. satisfying the condition:
\begin{equation}
  \label{eq:square-integrable}
  0< \norm{\tilde {v}_0}^2=\int_{G}
  \modulus{\scalar{v_0}{\uir{}{2}(g) v_0}}^2\,dg<\infty.
\end{equation}
In this case the wavelet (covariant) transform is a map into the
square integrable functions~\cite{DufloMoore} with respect to the left
Haar measure%
\index{invariant!measure}%
\index{measure!invariant} on \(G\). The map becomes an isometry if \(v_0\) is
properly scaled. Moreover, we are able to recover the input \(v\) from its
wavelet transform through the reconstruction formula, which requires an
admissible vector as well, see Example~\ref{ex:Haar-pairing} below.
The most popularized case of the above scheme is
provided by the affine group. 

\begin{example}
  \label{ex:ax+b}
  For the \(ax+b\) group, 
  representation~\eqref{eq:ax+b-repr-quasi-reg} is square integrable for
  \(p=2\). Any function \(v_0\), such that its Fourier transform
  \(\hat{v}_0(\lambda )\) satisfies 
  \begin{equation}
    \label{eq:ax+b-admiss-cond}
    \int\limits_0^\infty \frac{\modulus{\hat{v}_0(\lambda
        )}^2}{\lambda }\,d\lambda 
    < \infty,
  \end{equation}
  is admissible in the sense of
  \eqref{eq:square-integrable}~\cite{AliAntGaz00}*{\S~12.2}. The
  \emph{continuous wavelet transform}%
  \index{continuous!wavelet transform}%
  \index{wavelet!transform!continuous}%
  \index{transform!wavelet!continuous} is generated by
  representation~\eqref{eq:ax+b-repr-quasi-reg} acting on an
  admissible vector \(v_0\) in
  expression~\eqref{eq:wavelet-transf}. The image of a function from
  \(\FSpace{L}{2}(\Space{R}{})\) is a function on the upper half-plane
  square integrable with respect to the measure \(a^{-2}\,da\,db\).
  There are many examples~\cite{AliAntGaz00}*{\S~12.2} of useful
  admissible vectors, say, the \emph{Mexican hat} wavelet%
  \index{Mexican hat!wavelet}%
  \index{wavelet!Mexican hat}: \((1-x^2)e^{-x^2/2}\).  For
  sufficiently regular \(\hat{v}_0\) 
  admissibility~\eqref{eq:ax+b-admiss-cond} of \(v_0\) follows by a
  weaker condition
  \begin{equation}
    \label{eq:admissibility-weaker}
    \int_{\Space{R}{}}  v_0(x)\,dx=0.
  \end{equation}
We dedicate
  Section~\ref{sec:transport-norm} to isometric properties of this
  transform.
\end{example}

However, square integrable representations and admissible vectors do
not cover all interesting cases.
\begin{example}
  \label{ex:cauchy-integral}
  For the above \(G=\mathrm{Aff}\) and
  representation~\eqref{eq:ax+b-repr-quasi-reg}, we consider the operators
  \(F_{\pm}:\FSpace{L}{p}(\Space{R}{}) \rightarrow \Space{C}{}\)
  defined by:
  \begin{equation}
    \label{eq:cauchy-pm}
    F_{\pm}(f)=\frac{1}{\pi \rmi}\int_{\Space{R}{}} \frac{f(x)\,dx}{\rmi\mp x}.
  \end{equation}
  In \(\FSpace{L}{2}(\Space{R}{})\) we note that
  \(F_+(f)=\scalar{f}{c}\), where
  \(c(x)=\frac{1}{\pi\rmi}\frac{1}{\rmi+x}\). Computing the Fourier
  transform \(\hat{c}(\lambda)=\chi_{(0,+\infty)}(\lambda)\, e^{-\lambda }\), we see that
  \(\bar{c} \in \FSpace{H}{2}(\Space{R}{})\). Moreover, \(\hat{c}\)
  does not satisfy  admissibility
  condition~\eqref{eq:ax+b-admiss-cond} for
  representation~\eqref{eq:ax+b-repr-quasi-reg}.

  Then, covariant transform~\eqref{eq:coheret-transf-gen} is 
  Cauchy integral%
  \index{integral!Cauchy}%
  \index{Cauchy!integral}~\eqref{eq:cauchy-int} from \(\FSpace{L}{p}(\Space{R}{})\) to the
  space of functions \(\tilde{f}(a,b)\) such that
  \(a^{-\frac{1}{p}}\tilde{f}(a,b)\) is in the Hardy space%
  \index{space!Hardy}%
  \index{Hardy!space} on the upper/lower half-plane
  \(\FSpace{H}{p}(\Space[\pm]{R}{2})\)~\cite{Nikolski02a}*{\S~A.6.3}.
  Due to inadmissibility of \(c(x)\), the
  complex analysis become decoupled from the traditional wavelet
  theory. 

  Many important objects in harmonic analysis are generated by
  inadmissible mother wavelets like~\eqref{eq:cauchy-pm}. For example,
  the functionals \( P=\frac{1}{2}( F_+ +F_-)\) and \( Q=\frac{1}{2\rmi}( F_+
  -F_-)\) are defined by kernels:
  \begin{align}
    \label{eq:poisson-kernel}
    p(x)&=\frac{1}{2\pi \rmi}\left(\frac{1}{\rmi-x}-
    \frac{1}{\rmi+x}\right)=\frac{1}{\pi }\frac{1}{1+x^2},\\
    \label{eq:conj-poisson-kernel}
    q(x)&=-\frac{1}{2\pi }\left(\frac{1}{\rmi-x}-
    \frac{1}{\rmi+x}\right)=-\frac{1}{\pi }\frac{x}{1+x^2}
  \end{align}
  which are \emph{Poisson kernel}%
  \index{Poisson kernel}%
  \index{kernel!Poisson}%
  \index{integral!Poisson|see{Poisson kernel}}~\eqref{eq:poisson-int}
  and the \emph{conjugate Poisson kernel}%
  \index{Poisson kernel!conjugated}%
  \index{kernel!Poisson!conjugated}~\citelist{\cite{Grafakos08}*{\S~4.1}
    \cite{Garnett07a}*{\S~III.1} \cite{Koosis98a}*{Ch.~5}
    \cite{Nikolski02a}*{\S~A.5.3}}, respectively.  Another interesting
   non-admissible vector is the \emph{Gaussian}\index{Gaussian}
  \(e^{-x^2}\).
\end{example}
\begin{example}
  \label{ex:maximal-function}
  A step in a different direction is a consideration of
  non-linear operators. Take again the \(ax+b\) group%
  \index{$ax+b$ group}%
  \index{group!$ax+b$} and its
  representation~\eqref{eq:ax+b-repr-quasi-reg}. 
  We define \(F\) to be a homogeneous (but non-linear) functional
  \(V\rightarrow \Space[+]{R}{}\):
  \begin{equation}
    \label{eq:F-modulus-def}
    F_m (f) = \frac{1}{2}\int\limits_{-1}^1 \modulus{f(x)}\,dx.
  \end{equation}
  Covariant transform~\eqref{eq:coheret-transf-gen} becomes:
  \begin{equation}
    \label{eq:hardy-max}
    [\oper{W}^m_p f](a,b) =  F(\uir{}{p}({\textstyle\frac{1}{a},-\frac{1}{b}}) f) 
    = \frac{1}{2}\int\limits_{-1}^1
    \modulus{a^{\frac{1}{p}}f\left(ax+b\right)}\,dx
    = \frac{a^{\frac{1}{q}}}{2}\int\limits^{b+a}_{b-a}
    \modulus{f\left(x\right)}\,dx,
  \end{equation}
  where \(\frac{1}{p}+\frac{1}{q}=1\), as usual.
  We will see its connections with the Hardy--Littlewood maximal
  functions%
  \index{Hardy--Littlewood!maximal functions}%
  \index{maximal functions!Hardy--Littlewood} in
  Example~\ref{ex:hardy-littlewood}. 
\end{example}
Since linearity has clear advantages, we may prefer to reformulate the
last example using linear covariant transforms. The idea is similar
to the representation of a convex function as an envelope of linear
ones, cf.~\cite{Garnett07a}*{Ch.~I, Lem.~6.1}. To this end, we take a
collection \(\mathbf{F}\) of linear fiducial functionals and, for a
given function \(f\), consider the set of all covariant transforms
\(\oper{W}_F f\), \(F\in\mathbf{F}\).
\begin{example}
  Let us return to the setup of the previous Example for
  \(G=\mathrm{Aff}\) and its
  representation~\eqref{eq:ax+b-repr-quasi-reg}. Consider the unit ball \(B\)
  in \(\FSpace{L}{\infty}[-1,1]\). Then, any \(\omega\in B\) defines a
  bounded linear functional \(F_\omega\) on \(\FSpace{L}{1}(\Space{R}{})\):
  \begin{equation}
    \label{eq:hyperbolic-functional}
    F_\omega(f)= \frac{1}{2}\int\limits_{-1}^1 f(x)\, \omega(x)\,dx
    =  \frac{1}{2}\int_{\Space{R}{}} f(x)\, \omega(x)\,dx.
  \end{equation}
  Of course, \(\sup_{\omega\in B} F_\omega(f)= F_m(f)\) with \(F_m\)
  from~\eqref{eq:F-modulus-def} and for all
  \(f\in\FSpace{L}{1}(\Space{R}{})\). Then, for the non-linear
  covariant transform~\eqref{eq:hardy-max} 
  we have the following expression in terms of the linear covariant transforms
  generated by \(F_\omega\):
  \begin{equation}
    \label{eq:nl-through-sup}
    [\oper{W}^m_1 f](a,b) = \sup_{\omega\in B} \,  [\oper{W}^\omega _1
      f](a,b).
  \end{equation}
  The presence of suprimum is the price to pay for such a ``linearization''.
\end{example}
\begin{rem}
  \label{re:grand-maximal-funct}
  The above construction is not much different to the 
  \emph{grand maximal function}%
  \index{grand maximal function}%
  \index{maximal function!grand}~\cite{Stein93}*{\S~III.1.2}.
  Although, it may look like a generalisation of covariant transform,
  grand maximal function can be realised as a particular case of
  Defn.~\ref{de:covariant-transform}. Indeed, let \(M(V)\) be a
  subgroup of the group of all invertible isometries of a metric space
  \(V\). If \(\uir{}{}\) represents a group  \(G\) by isometries of \(V\)
  then we can consider the group \(\tilde{G}\) generated by all finite products of
  \(M(V)\) and \(\uir{}{}(g)\), \(g\in G\) with the straightforward
  action \(\uir[2]{}{}\) on \(V\). The grand maximal functions
  is produced by the covariant transform for the representation
  \(\uir[2]{}{}\) of \(\tilde{G}\).
\end{rem}
\begin{rem}
  \label{re:hypercomplex-numbers}
  It is instructive to compare action~\eqref{eq:sl2r-action-rl} of
  the large \(\SL\)\index{$\SL$}%
  \index{group!$\SL$} group on the mother wavelet \(\frac{1}{x+\rmi}\)
  for the Cauchy integral and the principal case
  \(\omega(x)=\chi_{[-1,1]}(x)\) (the characteristic function of
  \([-1,1]\)) for functional~\eqref{eq:hyperbolic-functional}. The
  wavelet \(\frac{1}{x+\rmi}\) is an eigenvector for all matrices
  \(\begin{pmatrix} \cos t & \sin t\\ -\sin t& \cos t
  \end{pmatrix}\), which form the one-parameter compact subgroup
  \(K\subset\SL\). The respective covariant transform (i.e., the
  Cauchy integral) maps functions to the homogeneous space \(\SL/K\),
  which is the upper half-plane with the M\"obius (linear-fractional)
  transformations of complex numbers%
  \index{complex numbers}%
  \index{numbers!complex}~\citelist{\cite{Kisil05a}
    \cite{Kisil12a}*{\S~3.3.4} \cite{Kisil11c}*{\S~3}}. By contrast,
  the mother wavelet \(\chi_{[-1,1]}\) is an eigenvector for all
  matrices \(\begin{pmatrix} \cosh t & \sinh t\\ \sinh t& \cosh t
  \end{pmatrix}\), which form the one-parameter subgroup
  \(A\in\SL\). The covariant transform (i.e., the averaging) maps
  functions to the homogeneous space \(\SL/A\), which can be
  identified with a set of double numbers%
  \index{double numbers}%
  \index{numbers!double} with corresponding M\"obius
  transformations~\citelist{\cite{Kisil05a} \cite{Kisil12a}*{\S~3.3.4}
    \cite{Kisil11c}*{\S~3}}. Conformal geometry of double numbers is
  suitable for real variables technique, in particular,
  tents\index{tent}~\cite{Stein93}*{\S~II.2} make a M\"obius-invariant
  family.
\end{rem}

\section{\large The contravariant transform}
\label{sec:invar-funct-groups}

Define the left action \(\Lambda\) of a group \(G\) on a space of
functions over \(G\) by:
\begin{equation}
  \label{eq:left-reg-repr}
  \Lambda(g): f(h) \mapsto f(g^{-1}h).
\end{equation}
For example, in the case of the affine group it is~\eqref{eq:ax+b-left-regular}.
An object invariant under the left action
\(\Lambda\) is called \emph{left invariant}.\index{invariant}
In particular, let \(L\) and \(L'\) be two left invariant spaces of
functions on \(G\).  We say that a pairing \(\scalar{\cdot}{\cdot}:
L\times L' \rightarrow \Space{C}{}\) is \emph{left invariant}%
\index{invariant!pairing}%
\index{pairing!invariant} if
\begin{equation}
  \label{eq:invariance-pairing}
  \scalar{\Lambda(g)f}{\Lambda(g) f'}= \scalar{f}{f'}, \quad \textrm{ for all }
  \quad f\in L,\  f'\in L', \ g\in G.
\end{equation}
\begin{rem}
  \begin{enumerate}
  \item We do not require the pairing to be linear in general, in some
    cases it is sufficient to have only homogeneity, see
    Example~\ref{ex:maximal-function-inv-wavelet}. 
  \item If the pairing is invariant on space \(L\times L'\) it is not
    necessarily invariant (or even defined) on large spaces  of
    functions.
  \item In some cases, an invariant pairing on \(G\) can be obtained
    from an \emph{invariant functional}%
    \index{invariant!functional}%
    \index{functional!invariant} \(l\) by the formula
    \(\scalar{f_1}{f_2}=l(f_1 {f}_2)\).
  \end{enumerate}
\end{rem}

For a representation \(\uir{}{}\) of \(G\) in \(V\) and \(w_0\in V\),
we construct a function \(w(g)=\uir{}{}(g)w_0\) on \(G\). We assume
that the pairing can be extended in its second component to this
\(V\)-valued functions. For example, such an extension can be defined
in the weak sense.
\begin{defn}\cites{Kisil09d,Kisil11c}
  \label{de:admissible}
  Let \(\scalar{\cdot}{\cdot}\) be a left invariant pairing on
  \(L\times L'\) as above, let \(\uir{}{}\) be a representation of
  \(G\) in a space \(V\), we define the function
  \(w(g)=\uir{}{}(g)w_0\) for \(w_0\in V\) such that \(w(g)\in L'\) in
  a suitable sense. The \emph{contravariant transform}%
  \index{contravariant!transform}%
  \index{transform!contravariant}%
  \index{inverse!covariant transform|see{contravariant transform}}%
  \index{covariant!transform!inverse|see{contravariant transform}}%
  \index{transform!covariant!inverse|see{contravariant transform}}
  \(\oper{M}_{w_0}^{\uir{}{}}\) is a map \(L \rightarrow V\) defined
  by the pairing:
  \begin{equation}
    \label{eq:inv-cov-trans}
    \oper{M}_{w_0}^{\uir{}{}}: f \mapsto \scalar{f}{w}, \qquad \text{
      where } f\in L. 
  \end{equation}
\end{defn} 
We can drop out sup/subscripts in \(\oper{M}_{w_0}^{\uir{}{}}\) as we
did for  \(\oper{W}_{F}^{\uir{}{}}\).

\begin{example}[Haar paring]
  \label{ex:Haar-pairing}
  The most used example of an invariant pairing on
  \(\FSpace{L}{2}(G,d\mu)\times \FSpace{L}{2}(G,d\mu)\)
  is the integration with respect to the Haar measure:%
  \index{Haar measure|see{invariant measure}}%
  \index{measure!Haar|see{invariant measure}}%
  \index{invariant!measure}%
  \index{measure!invariant}
   \begin{equation}
     \label{eq:haar-pairing}
     \scalar{f_1}{f_2}=\int_G f_1(g) {f}_2(g)\,dg.
   \end{equation}
   If \(\uir{}{}\) is a square integrable representation%
   \index{square integrable!representation}%
   \index{representation!square integrable} of \(G\) and \(w_0\) is an
   admissible vector%
   \index{admissible wavelet}%
   \index{wavelet!admissible}, see Example~\ref{ex:wavelet}, then this
   pairing can be extended to \(w(g)=\uir{}{}(g) w_0\). The 
   contravariant transform is known in this setup as the
   \emph{reconstruction formula}%
   \index{reconstruction formula}%
   \index{formula!reconstruction}, cf.~\cite{AliAntGaz00}*{(8.19)}:
   \begin{equation}
     \label{eq:wavelet-reconstruction}
     \oper{M}_{w_0} f =\int _G f(g)\, w(g)\,dg,
     \qquad\text{  where } w(g)=\uir{}{}(g) w_0.
   \end{equation}
   It is possible to use different admissible vectors \(v_0\) and
   \(w_0\) for wavelet
   transform~\eqref{eq:wavelet-transf} and reconstruction
   formula~\eqref{eq:wavelet-reconstruction}, respectively, cf.
   Example~\ref{ex:duflo-moor}.
 \end{example}
 

 Let either
 \begin{itemize}
 \item \(\uir{}{}\) be not a square integrable representation (even modulo
   a subgroup); \emph{or}
 \item \(w_0\) be an inadmissible vector of a square integrable
   representation \(\uir{}{}\).
 \end{itemize}
 A suitable invariant pairing in this case is not associated with 
 integration over the Haar measure on \(G\). In this
 case we speak about a \emph{Hardy pairing}%
\index{Hardy!pairing}%
\index{pairing!Hardy}. The following example explains
 the name.
\begin{example}[Hardy pairing]
  Let \(G\) be the \(ax+b\) group%
  \index{$ax+b$ group}%
  \index{group!$ax+b$} and its representation
  \(\uir{}{}\)~\eqref{eq:ax+b-repr-quasi-reg} in Example~\ref{ex:ax+b}. An
  invariant pairing on \(G\), which is not generated by the Haar
  measure%
   \index{invariant!measure}%
   \index{measure!invariant} \(a^{-2}da\,db\), is:
  \begin{equation}
    \label{eq:hardy-pairing}
    \scalar[H]{f_1}{f_2}=
    \lim_{a\rightarrow 0} \int\limits_{-\infty}^{\infty}
    f_1(a,b)\,{f}_2(a,b)\,\frac{db}{a}.
  \end{equation}
  For this pairing, we can consider functions \(\frac{1}{\pi \rmi}\frac{1}{
    x+\rmi}\) or \(e^{-x^2}\), which are not admissible vectors in the
  sense of square integrable representations. For example, for
  \(v_0=\frac{1}{\pi \rmi}\frac{1}{ x+\rmi}\) we obtain:
  \begin{displaymath}
    [\oper{M}f](x)=
    \lim_{a\rightarrow 0} \int\limits_{-\infty}^{\infty}
    f(a,b)\, \frac{a^{-\frac{1}{p}}}{\pi \rmi (x+\rmi a-b)} \,db
   =  -\lim_{a\rightarrow 0} \frac{a^{-\frac{1}{p}}}{\pi \rmi}\int\limits_{-\infty}^{\infty}
    \frac{f(a,b)\,db}{b-(x+\rmi a)} .
  \end{displaymath}
  In other words, it expresses the boundary values at \(a=0\) of the Cauchy integral
  \([-\oper{C}f](x+\rmi a)\).
\end{example}


Here is an important example of non-linear pairing.
\begin{example}
  \label{ex:maximal-function-inv-wavelet}
  Let \(G=\mathrm{Aff}\) and an invariant homogeneous functional
  on \(G\) be given by the \(\FSpace{L}{\infty}\)-version of Haar
  functional~\eqref{eq:haar-pairing}: 
  \begin{equation}
    \label{eq:sup-functional-Aff}
    \scalar[\infty]{f_1}{f_2}=\sup_{g\in G}\modulus{f_1(g){f}_2(g)}.
  \end{equation}
 Define the following two functions on \(\Space{R}{}\):
 \begin{equation}
   \label{eq:v-plus-v-star}
    {v}^+_0(t)=\left\{
      \begin{array}{ll}
        1,&\text{ if } t=0;\\
        0,&\text{ if } t\neq 0,
      \end{array}
    \right.\quad \text{ and }\quad
    v_0^*(t)=\left\{
      \begin{array}{ll}
        1,&\text{ if } \modulus{t}\leq 1;\\
        0,&\text{ if } \modulus{t}> 1 .
      \end{array}
    \right. 
 \end{equation}
 The respective contravariant transforms are generated by
 representation \(\uir{}{\infty}\)~\eqref{eq:ax+b-repr-quasi-reg} are:
 \begin{eqnarray}
   \label{eq:vert-maximal-function}
   [ \oper{M}_{{v}^+_0}f](t)&=&f^+(t)=\scalar[\infty]{f(a,b)}{\uir{}{\infty}(a,b)
     {v}_0^+(t)}=
   \sup_{a}\modulus{f(a,t)},\\{}
   \label{eq:nt-maximal-function}
   [ \oper{M}_{v^*_0}f](t)&=&f^*(t)=\scalar[\infty]{f(a,b)}{\uir{}{\infty}(a,b)
     v_0^*(t)}=
   \sup_{a>\modulus{b-t}}\modulus{f(a,b)}.
 \end{eqnarray}
 Transforms~\eqref{eq:vert-maximal-function}
 and~\eqref{eq:nt-maximal-function} are the \emph{vertical}%
  \index{vertical!maximal functions}%
  \index{maximal functions!vertical} and
 \emph{non-tangential maximal
   functions}%
  \index{non-tangential!maximal functions}%
  \index{maximal functions!non-tangential}~\cite{Koosis98a}*{\S~VIII.C.2}, respectively.
\end{example}
\begin{example}
  \label{ex:limit-function-inv-wavelet}
  Consider again \(G=\mathrm{Aff}\) equipped now with an
  invariant linear functional, which is a Hardy-type modification
  (cf.~\eqref{eq:hardy-pairing}) of
  \(\FSpace{L}{\infty}\)-functional~\eqref{eq:sup-functional-Aff}:
  \begin{equation}
    \label{eq:hardy-infinity}
    \scalar[\stackrel{H}{\infty}]{f_1}{f_2}={\varlimsup_{a\rightarrow
        0}}\,\sup_{b\in\Space{R}{}}(f_1(a,b) {f}_2(a,b)),
  \end{equation}
  where \(\varlimsup\) is the upper limit. Then, the covariant
  transform \(\oper{M}^H\) for this pairing for functions \(v^+\) and
  \(v^*\)~\eqref{eq:v-plus-v-star} becomes:
 \begin{eqnarray}
   \label{eq:norm-limit}
   [ \oper{M}_{{v}^+_0}^Hf](t)
   &=&\scalar[\stackrel{H}{\infty}]{f(a,b)}{\uir{}{\infty}(a,b)
     {v}_0^+(t)}=
   \varlimsup_{a\rightarrow 0}f(a,t),\\
   \label{eq:nt-limit}
   [ \oper{M}_{v^*_0}^Hf](t)
   &=&\scalar[\stackrel{H}{\infty}]{f(a,b)}{\uir{}{\infty}(a,b)
     v_0^*(t)}=
   \varlimsup_{
     \substack{
        a\rightarrow 0\\
        \modulus{b-t}<a}} f(a,b).
 \end{eqnarray}
  They are the \emph{normal} and
  \emph{non-tangential} upper limits from the upper-half plane to
  the real line, respectively.
\end{example}

Note the obvious inequality \(\scalar[\infty]{f_1}{f_2} \geq
\scalar[\stackrel{H}{\infty}]{f_1}{f_2}\) between
pairings~\eqref{eq:sup-functional-Aff} and~\eqref{eq:hardy-infinity},
which produces the corresponding relation between respective 
contravariant transforms.

There is an explicit duality between the covariant transform and the
contravariant transform. Discussion of the grand maximal function in the
Rem.~\ref{re:grand-maximal-funct} shows usefulness of the covariant
transform over a family of fiducial functionals. Thus, we shall not
be surprised by the contravariant transform over a family of
reconstructing vectors as well.
\begin{defn}
  Let \(w: \mathrm{Aff} \rightarrow \FSpace{L}{1}(\Space{R}{})\)
  be a function.  We define a new function \(\uir{}{1}w\) on
  \(\mathrm{Aff}\) with values in \(\FSpace{L}{1}(\Space{R}{})\)
  via the point-wise action \([\uir{}{1}
  w](g)=\uir{}{1}(g) w(g)\) of
  \(\uir{}{\infty}\)~\eqref{eq:ax+b-repr-quasi-reg}. If
  \(\sup_g\norm[1]{w(g)}< \infty\), then, for
  \(f\in\FSpace{L}{1}(\mathrm{Aff})\), we define the \emph{extended
    contravariant transform} by:
  \begin{equation}
    \label{eq:extend-inverse-covar-trans}
    [\oper{M}_w f](x)=\int_{\mathrm{Aff}} f(g)\, [\uir{}{1} w](g)\,dg.
  \end{equation}
\end{defn}
Note, that \eqref{eq:extend-inverse-covar-trans} reduces to the 
contravariant transform~\eqref{eq:wavelet-reconstruction} if we start from
the constant function \(w(g)=w_0\).
\begin{defn}
  We call a function \(r\) on \(\Space{R}{}\) a \emph{nucleus}%
  \index{nucleus} if:
  \begin{enumerate}
  \item \(r\) is supported in \([-1,1]\),
  \item \(\modulus{r}<\frac{1}{2}\) almost everywhere, and
  \item \(\int_{\Space{R}{}} r(x)\,dx=0\), cf.~\eqref{eq:admissibility-weaker}.
  \end{enumerate}
\end{defn}
Clearly, for a nucleus \(r\), the function \(s=\uir{}{1}(a,b) r\) has
the following properties:
\begin{enumerate}
\item \(s\) is supported in a ball centred at \(b\) and radius \(a\), 
\item \(\modulus{s}<\frac{1}{2a}\) almost everywhere, and 
\item \(\int_{\Space{R}{}} s(x)\,dx=0\).
\end{enumerate}
In other words, \(s=\uir{}{1}(a,b) r\) is an \emph{atom}%
\index{atom}, cf.~\cite{Stein93}*{\S~III.2.2} and any atom may be
obtained in this way from some nucleus and certain \((a,b)\in\mathrm{Aff}\).
\begin{example}
  \label{ex:atomic-decomposition}
  Let \(f(g)=\sum_j \lambda_j \delta_{g_j}(g)\) with \(\sum_j
  \modulus{\lambda_j}<\infty\) be a countable sum of point masses on
  \(\mathrm{Aff}\). If all values of \(w(g_j)\) are nucleuses,
  then~\eqref{eq:extend-inverse-covar-trans} becomes:
  \begin{equation}
    \label{eq:atomic-decomp-gen}
    [\oper{M}_w f](x)=\int_{\mathrm{Aff}} f(g)\, [\uir{}{1} w](g)\,dg
    =\sum_j \lambda_j s_j,
  \end{equation}
  where \(s_j=\uir{}{1}(g_j) w(g_j)\) are atoms.
 The right-hand side of \eqref{eq:atomic-decomp-gen} is known as an \emph{atomic
  decomposition}%
\index{atomic!decomposition}%
\index{decomposition!atomic} of a function \(h(x)=[\oper{M}_w
  f](x)\), see~\cite{Stein93}*{\S~III.2.2}.
\end{example}

\section{\large Intertwining properties of covariant transforms}
\label{sec:interw-prop-covar}

The covariant transform has obtained its name because of the following property.
\begin{thm}\cites{Kisil09d,Kisil11c}
  \label{pr:inter1} 
  Covariant transform~\eqref{eq:coheret-transf-gen}
  intertwines%
  \index{intertwining operator}%
  \index{operator!intertwining} \(\uir{}{}\) and the left regular representation
  \(\Lambda\)~\eqref{eq:left-reg-repr}  on \(\FSpace{L}{}(G,U)\):
  \begin{equation}
    \label{eq:cov-trans-intertwine}
    \oper{W} \uir{}{}(g) = \Lambda(g) \oper{W}.
  \end{equation}
\end{thm}
\begin{cor}\label{co:pi}
  The image space \(\oper{W}(V)\) is invariant under the
  left shifts on \(G\).
\end{cor}
The covariant transform is also a natural source of \emph{relative
  convolutions}%
\index{relative!convolution}%
\index{convolution!relative}~\cites{Kisil94e,Kisil13a}, which are
operators \(A_k=\int_G k(g)\uir{}{}(g)\,dg\) obtained by integration a
representation \(\uir{}{}\) of a group \(G\) with a suitable kernel
\(k\) on \(G\). In particular, inverse wavelet transform
\(\oper{M}_{w_0} f\)~\eqref{eq:wavelet-reconstruction} can be defined
from the relative convolution \(A_f\) as well: \(\oper{M}_{w_0} f= A_f
w_0\).
\begin{cor}
  Covariant transform~\eqref{eq:coheret-transf-gen} intertwines
  the operator of convolution \(K\) (with kernel \(k\)) and the operator
  of relative convolution \(A_k\), i.e. \(K \oper{W}= \oper{W} A_k\).
\end{cor}

If the invariant pairing is defined by integration with respect to the Haar
measure, cf. Example~\ref{ex:Haar-pairing}, then we can show an
intertwining property for the contravariant transform as well. 
\begin{prop}\cite{Kisil98a}*{Prop.~2.9}
  \label{pr:intertw-inverse-tr}
  Inverse wavelet transform \(
  \oper{M}_{w_0}\)~\eqref{eq:wavelet-reconstruction} intertwines 
  left regular representation \( \Lambda \)~\eqref{eq:left-reg-repr}
  on \( \FSpace{L}{2}(G)\) and \( \uir{}{} \):
  \begin{equation}
    \label{eq:inv-transform-intertwine}
    \oper{M}_{w_0} \Lambda(g) = \uir{}{}(g) \oper{M}_{w_0}.
  \end{equation}
\end{prop}
\begin{cor}\label{co:lambda}
  The image \(\oper{M}_{w_0}(\FSpace{L}{}(G))\subset V\) of a left
  invariant space \(\FSpace{L}{}(G)\) under the inverse wavelet
  transform \(\oper{M}_{w_0}\) is invariant under the representation
  \(\uir{}{}\).
\end{cor}
\begin{rem}
  \label{re:intertw-inverse-ext}
  It is an important observation, that the above intertwining property
  is also true for some contravariant transforms which are not based
  on pairing~\eqref{eq:haar-pairing}. For example, in the case of the
  affine group all pairings~\eqref{eq:hardy-pairing},
  \eqref{eq:hardy-infinity} and
  (non-linear!)~\eqref{eq:sup-functional-Aff} satisfy
  to~\eqref{eq:inv-transform-intertwine} for the respective representation
  \(\uir{}{p}\)~\eqref{eq:ax+b-repr-quasi-reg}. 
\end{rem}
There is also a simple connection between a covariant transform and
right shifts.
\begin{prop}\cites{Kisil10c,Kisil11c}
  \label{pr:inducer-wave-intertw}
  Let \(G\) be a Lie group and \(\uir{}{}\) be a representation of
  \(G\) in a space \(V\). Let \([\oper{W}f](g)=F(\uir{}{}(g^{-1})f)\) be a
  covariant transform defined by a fiducial operator \(F: V \rightarrow U\).
  Then the right shift \([\oper{W}f](gg')\) by \(g'\) is the covariant transform
  \([\oper{W'}f](g)=F'(\uir{}{}(g^{-1})f)]\) defined by the fiducial operator
  \(F'=F\circ\uir{}{}(g^{-1})\). 

  In other words the covariant transform intertwines%
  \index{intertwining operator}%
  \index{operator!intertwining} right shifts \(R(g): f(h) \mapsto
  f(hg)\) on the group \(G\) with the associated action 
  \begin{equation}
    \label{eq:fiducial-action}
    \uir{}{B}(g): F\mapsto F\circ\uir{}{}(g^{-1})
  \end{equation}
  on fiducial operators:
  \begin{equation}
    \label{eq:cov-trans-intertw-right}
    R(g) \circ \oper{W}_{F}=\oper{W}_{\uir{}{B}(g)F}, \qquad g\in G.
  \end{equation}
\end{prop}
Although the above result is obvious, its infinitesimal version has
interesting consequences.  Let \(G\) be a Lie group with a Lie algebra
\(\algebra{g}\) and \(\uir{}{}\) be a smooth representation of \(G\).
We denote by \(d\uir{}{B}\) the derived representation of the
associated representation \(\uir{}{B}\)~\eqref{eq:fiducial-action} on
fiducial operators.
\begin{cor}\cites{Kisil10c,Kisil11c}
  \label{co:cauchy-riemann}
  Let a fiducial operator \(F\) be a null-solution, i.e. \(A F=0\),
  for the operator \(A=\sum_j a_j d\uir{X_j}{B}\), where
  \(X_j\in\algebra{g}\) and \(a_j\) are constants.  Then the covariant
  transform \([\oper{W}_F f](g)=F(\uir{}{}(g^{-1})f)\) for any \(f\)
  satisfies
  \begin{displaymath}
    D (\oper{W}_F f)= 0, \qquad \text{where} \quad
    D=\sum_j \bar{a}_j\linv{X_j}.
  \end{displaymath}
  Here, \(\linv{X_j}\) are the left invariant fields (Lie derivatives) on
  \(G\) corresponding to \(X_j\).
\end{cor}
\begin{example}
  \label{ex:cauchy-riemann-wav}
  Consider representation \(\uir{}{}\)~\eqref{eq:ax+b-repr-quasi-reg} of
  the \(ax+b\) group with the \(p=1\). Let \(\mathsf{A}\) and \(\mathsf{N}\) be the
  basis of \(\algebra{g}\) generating one-parameter subgroups \(A\)
  and \(N\)~\eqref{eq:subgroups-A-N}, respectively. Then, the derived representations are:
  \begin{displaymath}
    [d\uir{\mathsf{A}}{} f](x)= -f(x)-xf'(x), \qquad [d\uir{\mathsf{N}}{}f](x)=-f'(x).
  \end{displaymath}
  The corresponding left invariant vector fields on \(ax+b\) group%
  \index{$ax+b$ group}%
  \index{group!$ax+b$} are:
  \begin{displaymath}
   \linv{\mathsf{A}} =a \partial_a,\qquad \linv{\mathsf{N}}=a\partial_b.
  \end{displaymath}
  The mother wavelet \(\frac{1}{x+\rmi}\) in~\eqref{eq:cauchy-pm} is a
  null solution of the operator
  \begin{equation}
    -d\uir{\mathsf{A}}{} -\rmi d\uir{\mathsf{N}}{}=I+(x+\rmi)\frac{d}{dx}.
  \end{equation}
  Therefore, the image of the covariant transform with fiducial
  operator \(F_+\)~\eqref{eq:cauchy-pm} consists of the null solutions
  to the operator \(-\linv{\mathsf{A}}+\rmi\linv{\mathsf{N}}=\rmi
  a(\partial_b+\rmi\partial_a)\), that is in the essence 
  Cauchy--Riemann operator%
  \index{Cauchy--Riemann operator}%
  \index{operator!Cauchy--Riemann}
  \(\partial_{\bar{z}}\)~\eqref{eq:CR-Laplace} in the upper
  half-plane.
\end{example}
\begin{example}
  \label{ex:laplace-wavelet}
  In the above setting, the function
  \(p(x)=\frac{1}{\pi}\frac{1}{x^2+1}\)~\eqref{eq:poisson-kernel} is a
  null solution of the operator:
  \begin{displaymath}
    (d\uir{\mathsf{A}}{})^2 - d\uir{\mathsf{A}}{}
    +(d\uir{\mathsf{N}}{})^2
    =2I+4x\frac{d}{dx}+(1+x^2)\frac{d^2}{dx^2}.
  \end{displaymath}
  The covariant transform with the mother wavelet \(p(x)\) is the
  Poisson integral, its values are null solutions to the operator
  \((\linv{\mathsf{A}})^2-\linv{\mathsf{A}}+(\linv{\mathsf{N}})^2
  =a^2(\partial_b^2+\partial_a^2)\), which is
  Laplace operator \(\Delta\)~\eqref{eq:CR-Laplace}.%
  \index{Laplace!operator}%
  \index{operator!Laplace}
\end{example}
\begin{example}
  \label{ex:dyadic-cubes-wavelet}
  Fiducial functional \(F_m\)~\eqref{eq:F-modulus-def} is a null
  solution of the following functional equation:
  \begin{displaymath}
    \textstyle
    F_m-F_m\circ\uir{}{\infty}(\frac{1}{2},\frac{1}{2})-F_m\circ\uir{}{\infty}(\frac{1}{2},-\frac{1}{2})=0.
  \end{displaymath}
  Consequently, the image of wavelet transform
  \(\oper{W}^m_p\)~\eqref{eq:hardy-max} consists of functions which
  solve the equation:
  \begin{displaymath}
    \textstyle
    (I-R (\frac{1}{2},\frac{1}{2})-R (\frac{1}{2},-\frac{1}{2}))f=0
    \quad\text{ or }\quad
    f(a,b)=f(\frac{1}{2} a, b+\frac{1}{2}a)+f(\frac{1}{2} a, b-\frac{1}{2}a).
  \end{displaymath}
  The last relation is the key to the stopping time
  argument%
  \index{stopping time argument}~\cite{Garnett07a}*{Ch.~VI, Lem.~2.2}
  and the dyadic squares%
  \index{dyadic!squares}%
  \index{squares!dyadic} technique, see for
  example~\cite{Stein93}*{\S~IV.3}, \cite{Garnett07a}*{Ch.~VII, Thm.~1.1} or the picture on the
  front cover of the latter book.
\end{example}
The moral of the above
Examples~\ref{ex:cauchy-riemann-wav}--\ref{ex:dyadic-cubes-wavelet}
is: there is a significant freedom in choice of covariant
transforms. However, some fiducial functionals have special properties,
which suggest the suitable technique (e.g., analytic, harmonic,
dyadic, etc.)  following from this choice.


\section{\large Composing the covariant and the contravariant transforms}
\label{sec:comp-covar-transfr}

From Props.~\ref{pr:inter1}, \ref{pr:intertw-inverse-tr} and
Rem.~\ref{re:intertw-inverse-ext} we deduce the following
\begin{cor}
  The composition \(\oper{M}_w \circ \oper{W}_F\) of a covariant
  \(\oper{M}_w \) and contravariant \(\oper{W}_F\) transforms is a
  map \(V\rightarrow V\), which commutes with \(\uir{}{}\), i.e.,
  intertwines \(\uir{}{}\) with itself.

  In particular for the affine group and 
  representation~\eqref{eq:ax+b-repr-quasi-reg}, \(\oper{M}_w \circ
  \oper{W}_F\) commutes with shifts and dilations of the real line.
\end{cor}
Since the image space of \(\oper{M}_w \circ \oper{W}_F\) is an
\(\mathrm{Aff}\)-invariant space, we shall be interested in the smallest
building blocks with the same property. For the Hilbert spaces, any
group invariant subspace \(V\) can be decomposed into a direct
integral \(V=\oplus\int V_\mu\,d\mu\) of \emph{irreducible}%
\index{representation!irreducible}%
\index{irreducible!representation} subspaces \(V_\mu\), i.e. \(V_\mu\)
does not have any non-trivial invariant
subspace~\cite{Kirillov76}*{\S~8.4}.  For representations in Banach
spaces complete reducibility may not occur and we shall look for
\emph{primary} subspace, i.e. space which is not a direct sum of two
invariant subspaces~\cite{Kirillov76}*{\S~8.3}. We already identified
such subspaces as generalized Hardy spaces%
\index{space!Hardy!generalized}%
\index{Hardy!space!generalized} in Defn.~\ref{de:hardy-space}. They
are also related to covariant functional
calculus%
\index{covariant!functional calculus}%
\index{functional!calculus!covariant}%
\index{calculus!functional!covariant}~\citelist{\cite{Kisil02a} \cite{Kisil11c}*{\S~6}}.

For irreducible Hardy spaces, we can use the following general
principle, which has several different formulations,
cf.~\cite{Kirillov76}*{Thm.~8.2.1}:
\begin{lem}[Schur] 
  \label{le:Schur}
  \cite{AliAntGaz00}*{Lem.~4.3.1}
  Let \(\uir{}{}\) be a continuous unitary irreducible representation
  of \(G\) on the Hilbert space \(H\). If a bounded operator \(T: H
  \rightarrow T\) commutes
  with \(\uir{}{}(g)\), for all \(g \in G\), 
  then \(T = k I\), for some \(\lambda \in \Space{C}{}\).
\end{lem}
\begin{rem}
  A revision of proofs of the Schur's Lemma, even in different
  formulations, show that the result is related to the existence of
  joint invariant subspaces%
  \index{invariant!subspace}%
  \index{subspace!invariant} for all operators \(\uir{}{}(g)\), \(g\in G\).
\end{rem}

In the case of  classical wavelets,
the relation between wavelet transform~\eqref{eq:wavelet-transf}
and inverse wavelet transform~\eqref{eq:wavelet-reconstruction} is
suggested by their names. 
\begin{example}
  \label{ex:duflo-moor}
  For an irreducible square integrable representation and admissible
  vectors \(v_0\) and \(w_0\), there is the
  relation~\cite{AliAntGaz00}*{(8.52)}:
  \begin{equation}
    \label{eq:wave-trans-inverse-ident}
    \oper{M}_{w_0}\oper{W}_{v_0}=k I,
  \end{equation}
  as an immediate consequence from the Schur's lemma. Furthermore, 
  square integrability condition~\eqref{eq:square-integrable} ensures
  that \(k\neq 0\). The exact value of the constant \(k\) depends on
  \(v_0\), \(w_0\) and the Duflo--Moore
  operator~\citelist{\cite{DufloMoore} \cite{AliAntGaz00}*{\S~8.2}}.
\end{example}

It is of interest here, that two different vectors can be used as
analysing vector in~\eqref{eq:wavelet-transf} and for the
reconstructing formula~\eqref{eq:wavelet-reconstruction}. Even a
greater variety can be achieved if we use additional fiducial
operators and invariant pairings.

For the affine group, recall the decomposition from
Prop.~\ref{pr:fourier-decompos} into invariant subspaces
\(\FSpace{L}{2}(\Space{R}{})=\FSpace{H}{2}\oplus
\FSpace[\perp]{H}{2}\) and the fact, that the restrictions
\(\uir{+}{2}\) and \(\uir{-}{2}\) of
\(\uir{}{2}\)~\eqref{eq:ax+b-repr-quasi-reg} on \(\FSpace{H}{2}\) and
\(\FSpace[\perp]{H}{2}\) are not unitary equivalent. Then, Schur's
lemma implies:
\begin{cor}
  \label{co:intertwining-L2-H2}
  Any bounded linear operator \(T:
  \FSpace{L}{2}(\Space{R}{})\rightarrow \FSpace{L}{2}(\Space{R}{})\)
  commuting with \(\uir{}{2}\) has the form \(k_1
  I_{\FSpace{H}{2}} \oplus k_2 I_{\FSpace[\perp]{H}{2}}\) for some
  constants \(k_1\), \(k_2\in\Space{C}{}\). Consequently, the Fourier
  transform maps \(T\) to the operator of multiplication by
  \(k_1\chi_{(0,+\infty)}+k_2\chi_{(-\infty,0)}\).
\end{cor}
Of course, Corollary~\ref{co:intertwining-L2-H2} is applicable to the
composition of covariant and contravariant transforms. In
particular, the constants \(k_1\) and \(k_2\) may have zero values:
for example, the zero value occurs for
\(\oper{W}\)~\eqref{eq:wavelet-transf} with an admissible vector
\(v_0\) and non-tangential limit
\(\oper{M}_{v^*_0}^H\)~\eqref{eq:nt-limit}---because a square
integrable function \(f(a,b)\) on \(\mathrm{Aff}\) vanishes for
\(a\rightarrow 0\).

\begin{example}
  \label{ex:hardy-littlewood}
  The composition of contravariant transform
  \(\oper{M}_{v^*_0}\)~\eqref{eq:nt-maximal-function} with 
  covariant transform \(\oper{W}_\infty\)~\eqref{eq:hardy-max} is:
  \begin{eqnarray}
    \label{eq:hardy-littlewood-comp}
    [\oper{M}_{v^*_0} \oper{W}_\infty f](t)&=& 
    \sup_{a>\modulus{b-t}}\left\{ \frac{1}{2a} \int\limits^{b+a}_{b-a}
    \modulus{f\left(x\right)}\,dx\right\}\\
  &=&
        \sup_{b_1<t<b_2} \left\{\frac{1}{b_2-b_1} \int\limits^{b_2}_{b_1}
    \modulus{f\left(x\right)}\,dx\right\}.\nonumber 
  \end{eqnarray}
  Thus, \(\oper{M}_{v^*_0} \oper{W}_\infty f\) coincides with 
  \emph{Hardy--Littlewood maximal function}%
  \index{Hardy--Littlewood!maximal functions}%
  \index{maximal functions!Hardy--Littlewood}
  \(f^M\)~\eqref{eq:hardy-littlewood}, which contains important
  information on the original function \(f\)
  \cite{Koosis98a}*{\S~VIII.B.1}.  Combining Props.~\ref{pr:inter1}
  and~\ref{pr:intertw-inverse-tr} (through
  Rem.~\ref{re:intertw-inverse-ext}), we deduce that the operator \(M:
  f\mapsto f^M\) commutes%
  \index{intertwining operator}%
  \index{operator!intertwining} with \(\uir{}{p}\): \(\uir{}{p}M=M
  \uir{}{p}\).  Yet, \(M\) is non-linear and
  Cor.~\ref{co:intertwining-L2-H2} is not applicable in this case.
\end{example}

\begin{example}
  Let the mother wavelet \(v_0(x)=\delta(x)\) be the Dirac delta
  function, then the wavelet transform \(\oper{W}_\delta\) generated
  by \(\uir{}{\infty}\)~\eqref{eq:ax+b-repr-quasi-reg} on
  \(\FSpace{C}{}(\Space{R}{})\) is \([\oper{W}_\delta f](a,b)=f(b)\).
  Take the reconstruction vector \(w_0(t)=(1-\chi_{[-1,1]}(t))/t/\pi\) and
  consider the respective inverse wavelet transform \(\oper{M}_{w_0}\)
  produced by Hardy pairing~\eqref{eq:hardy-pairing}. Then, the
  composition of both maps is:
  \begin{eqnarray}
    [\oper{M}_{w_0}\circ \oper{W}_\delta f](t)&=&
        \lim_{a\rightarrow 0}\, \frac{1}{\pi}\!\int\limits_{-\infty}^{\infty}
    f(b)\,\uir{}{\infty}(a,b)w_0(t)\,\frac{db}{a} \nonumber \\
    &=& \lim_{a\rightarrow 0}\,  \frac{1}{\pi}\!\int\limits_{-\infty}^{\infty}
    f(b)\, \frac{1-\chi_{[-a,a]}(t-b)}{t-b} \,{db}\nonumber \\
    &=&\lim_{a\rightarrow 0}\,  \frac{1}{\pi}\!\int\limits_{\modulus{b}>a}
    \frac{f(b)}{t-b} \,{db}.\label{eq:hilbert-transform}
  \end{eqnarray}
  The last expression is the \emph{Hilbert transform}%
  \index{Hilbert!transform}%
  \index{transform!Hilbert} \(\oper{H}=\oper{M}_{w_0}\circ
  \oper{W}_\delta\), which is an example of a \emph{singular integral
    operator}%
  \index{singular!integral operator}%
  \index{operator!integral!singular} (SIO%
  \index{SIO@{see}singular integral
    operator})~\citelist{\cite{Stein93}*{\S~I.5}
    \cite{MityushevRogosin00a}*{\S~2.6}} defined through the principal
  value%
  \index{principal!value}%
  \index{value!principal}~\eqref{eq:pv-cauchy} (in the sense of Cauchy).  By
  Cor.~\ref{co:intertwining-L2-H2} we know that \(\oper{H}=k_1
  I_{\FSpace{H}{2}} \oplus k_2 I_{\FSpace[\perp]{H}{2}}\) for some
  constants \(k_1\), \(k_2\in\Space{C}{}\). Furthermore, we can
  directly check that \(\oper{H} J= -J\oper{H} \), for the reflection \(J\)
  from~\eqref{eq:J-map}, thus \(k_1=-k_2\). An evaluation of
  \(\oper{H}\) on a simple function from \(\FSpace{H}{2}\) (say, the
  Cauchy kernel \(\frac{1}{x+\rmi}\)) gives the value of the constant
  \(k_1=-\rmi\). Thus, \(\oper{H}=(-\rmi I_{\FSpace{H}{2}}) \oplus
  (\rmi I_{\FSpace[\perp]{H}{2}})\).
\end{example}
In fact, the previous reasons imply the following 
\begin{prop} 
  \label{pr:Hilbert-transf-charact}
  \cite{Stein70a}*{\S~III.1.1} Any bounded linear operator on
  \(\FSpace{L}{2}(\Space{R}{})\) commuting with quasi-regular
  representation \(\uir{}{2}\)~\eqref{eq:ax+b-repr-quasi-reg} and
  anticommuting with reflection \(J\)~\eqref{eq:J-map} is a
  constant multiple of Hilbert transform~\eqref{eq:hilbert-transform}.
\end{prop}
\begin{example}
  \label{ex:hilbert-trans-analytic}
  Consider the covariant transform \(\oper{W}_q\) defined by the inadmissible
  wavelet \(q(t)\)~\eqref{eq:conj-poisson-kernel}, the conjugated
  Poisson kernel%
  \index{Poisson kernel!conjugated}%
  \index{kernel!Poisson!conjugated}.  Its composition with the 
  contravariant transform \(\oper{M}_{{v}^+_0}^H\)~\eqref{eq:norm-limit}
  is
  \begin{equation}
    \label{eq:hilbert-trans-analytic}
    [\oper{M}_{{v}^+_0}^H\circ \oper{W}_q f](t)=
    \varlimsup_{a\rightarrow 0}
    \frac{1}{\pi}\int_{\Space {R}{}} \frac{f(x)\,(t-x)}{(t-x)^2+a^2}\,dx
  \end{equation}
  We can see that this composition satisfies to
  Prop.~\ref{pr:Hilbert-transf-charact}, the constant factor can again  be
  evaluated from the Cauchy kernel \(f(x)=\frac{1}{x+\rmi}\) and
  is equal to \(1\). Of course, this is a classical
  result~\cite{Grafakos08}*{Thm.~4.1.5} in harmonic analysis
  that~\eqref{eq:hilbert-trans-analytic} provides an alternative
  expression for Hilbert transform%
  \index{Hilbert!transform}%
  \index{transform!Hilbert}~\eqref{eq:hilbert-transform}.
\end{example}

\begin{example}
  Let \(\oper{W}\) be a covariant transfrom generated either by the
  functional \(F_\pm\)~\eqref{eq:cauchy-pm} (i.e. the Cauchy integral)
  or \(\frac{1}{2}  (F_+ -F_-)\) (i.e. the Poisson integral) from the
  Example~\ref{ex:cauchy-integral}. Then, for contravariant
  transform \(\oper{M}_{v^+_0}^H\)~\eqref{eq:vert-maximal-function}
  the composition \(\oper{M}_{v^+_0}^H \oper{W}\) becomes the normal
  boundary value of the Cauchy/Poisson integral, respectively.  The
  similar composition \(\oper{M}_{v^*_0}^H \oper{W}\) for
  reconstructing vector \(v^*_0\)~\eqref{eq:v-plus-v-star} turns to be
  the non-tangential limit of the Cauchy/Poisson integrals.
\end{example}

The maximal function and SIO are often treated as elementary
building blocks of harmonic analysis. In particular, it is common to
define the Hardy space%
\index{space!Hardy}%
\index{Hardy!space} as a closed subspace of
\(\FSpace{L}{p}(\Space{R}{})\) which is mapped to
\(\FSpace{L}{p}(\Space{R}{})\) by either the maximal
operator~\eqref{eq:hardy-littlewood-comp} or by the
SIO~\eqref{eq:hilbert-transform}~\citelist{\cite{Stein93}*{\S~III.1.2
    and \S~III.4.3} \cite{DziubanskiPreisner10a}}. From this
perspective, the coincidence of both characterizations seems to be
non-trivial. On the contrast, we presented both the maximal operator
and SIO as compositions of certain co- and contravariant
transforms. Thus, these operators act between certain
\(\mathrm{Aff}\)-invariant subspaces, which we associated with
generalized Hardy spaces%
\index{space!Hardy!generalized}%
\index{Hardy!space!generalized} in Defn.~\ref{de:hardy-space}. For the
right choice of fiducial functionals, the coincidence of the
respective invariant subspaces is quite natural.

The potential of the group-theoretical approach is not limited to the
Hilbert space \(\FSpace{L}{2}(\Space{R}{})\).  One of possibilities is
to look for a suitable modification of Schur's Lemma~\ref{le:Schur},
say, to Banach spaces. However, we can proceed with the affine group
without such a generalisation. Here is an illustration to a classical
question of harmonic analysis: to identify the class of functions on
the real line such that \(\oper{M}_{v^*_0}^H \oper{W}\) becomes the
identity operator on it.

\begin{prop}  
  \label{pr:nt-limit}
  Let \(\FSpace{B}{}\)
  be the space of bounded uniformly continuous functions
  on the real line. Let \(F: \FSpace{B}{}\rightarrow \Space{R}{}\) be
  a fiducial functional such that:
  \begin{equation}
    \label{eq:zero-limit-fiducial}
    \lim_{a\rightarrow 0} F(
    \uir{}{\infty}(1/a,0) f )= 0, \quad \text{ for all } f\in
    \FSpace{B}{}\text{ such that } f(0)=0 
  \end{equation}
  and \(F(\uir{}{\infty}(1,b) f)\) is a continuous function of
  \(b\in\Space{R}{}\) for a given \(f\in\FSpace{B}{}\).
  
  Then, \(\oper{M}_{v^*_0}^H\circ \oper{W}_F\) is a constant multiple
  of the identity operator on \(\FSpace{B}{}\).
\end{prop}
\begin{proof}
  First of all we note that \(\oper{M}_{v^+_0}^H \oper{W}_F\) is a
  bounded operator on \(\FSpace{B}{}\).  Let
  \(v^*_{(a,b)}=\uir{}{\infty}(a,b) v^*\). Obviously,
  \(v^*_{(a,b)}(0)=v^*(-\frac{b}{a})\) is an eigenfunction for
  operators \(\Lambda(a',0)\), \(a'\in\Space[+]{R}{}\) of the left
  regular representation of \(\mathrm{Aff}\):
  \begin{equation}
    \label{eq:v-eigenfunctions1}
    \Lambda(a',0) v^*_{(a,b)}(0)= v^*_{(a,b)}(0).
  \end{equation}
  This and the left invariance of 
  pairing~\eqref{eq:invariance-pairing} imply that
  \(\oper{M}_{v^*_0}^H\circ \Lambda (1/a,0)=\oper{M}_{v^*_0}^H\) for
  any \((a,0)\in\mathrm{Aff}\). Then, applying intertwining
  properties~\eqref{eq:cov-trans-intertwine} we obtain that
  \begin{eqnarray*}
    [\oper{M}_{v^*_0}^H \circ \oper{W}_F f](0)
    &=& [\oper{M}_{v^*_0}^H\circ  \Lambda (1/a,0)\circ \oper{W}_F f](0)\\
    &=& [\oper{M}_{v^*_0}^H\circ  \oper{W}_{F} \circ \uir{}{\infty}(1/a,0)f](0).
  \end{eqnarray*}
  Using the limit \(a\rightarrow 0\)~\eqref{eq:zero-limit-fiducial}
  and the continuity of \(F\circ \uir{}{\infty}(1,b)\) we conclude that
  the linear functional \(l:f\mapsto [\oper{M}_{v^*_0}^H\circ \oper{W}_F
  f](0)\) vanishes for any \(f\in\Space{B}{}\) such that \(f(0)=0\). 
  Take a function \(f_1\in\FSpace{B}{}\) such that \(f_1(0)=1\) and
  define \(c=l(f_1)\). From linearity of \(l\), for any \(f\in
  \FSpace{B}{}\) we have:
  \begin{displaymath}
    l(f)=l(f-f(0)f_1+f(0)f_1)=l(f-f(0)f_1)+f(0)l(f_1)=cf(0).
  \end{displaymath}

  Furthermore, using intertwining
  properties~\eqref{eq:cov-trans-intertwine} and~\eqref{eq:inv-transform-intertwine}:
  \begin{eqnarray*}
    [\oper{M}_{v^*_0}^H\circ \oper{W}_F
  f](t)&=&  [\uir{}{\infty}(1,-t) \circ \oper{M}_{v^*_0}^H \circ \oper{W}_F
  f](0)\\
  &=&  [\oper{M}_{v^*_0}^H \circ \oper{W}_F \circ \uir{}{\infty}(1,-t)   f](0)\\
  &=&  l ( \uir{}{\infty}(1,-t)   f)\\
  &=&  c[ \uir{}{\infty}(1,-t)   f](0)\\
  &=&  cf(t).
  \end{eqnarray*}
  This completes the proof.
\end{proof}

To get the classical statement we need the following lemma.
\begin{lem}
  For \(w(t)\in\FSpace{L}{1}(\Space{R}{})\), define
  the fiducial functional on \(\FSpace {B}{}\):
  \begin{equation}
    \label{eq:fiducial-from-wavelet}
    F(f)=\int_{\Space{R}{}} f(t)\,w(t)\, dt.
  \end{equation}
  Then  \(F\) satisfies the conditions (and thus the conclusions) of
  Prop.~\ref{pr:nt-limit}.
\end{lem}
\begin{proof}
  Let \(f\) be a continuous bounded function such that
  \(f(0)=0\). For \(\varepsilon>0\) chose 
  \begin{itemize}
  \item   \(\delta>0\) such that  \(\modulus{f(t)}<\varepsilon\) for
    all \(\modulus{t}<\delta\);
  \item \(M>0\) such that
    \(\int_{\modulus{t}>M}\modulus{w(t)}\,dt<\varepsilon\).
  \end{itemize}
  Then, for  \(a<\delta/M\), we have the estimation: 
  \begin{eqnarray*}
    \modulus{F(\uir{}{\infty}(1/a,0) f )}&=&
    \modulus{\int_{\Space{R}{}} f\left(at\right)\,w(t)\, dt}\\
    &\leq&    \modulus{\int_{\modulus{t}<M } f\left(at\right)\,w(t)\, dt}
    +    \modulus{\int_{\modulus{t}>M } f\left(at\right)\,w(t)\, dt}\\
    &\leq&    \varepsilon (\norm[1]{w}    +   \norm[\infty]{f}) .
  \end{eqnarray*}
  Finally, for a uniformly continuous function \(g\) for
  \(\varepsilon>0\) there is \(\delta>0\) such that
  \(\modulus{g(t+b)-g(t)}<\varepsilon\) for all \(b<\delta\) and
  \(t\in\Space{R}{}\). Then
  \begin{displaymath}
    \modulus{F(\uir{}{\infty}(1,b) g )-F(g)}=\modulus{\int_{\Space{R}{}}
    (g(t+b)-g(t))\,w(t)\, dt}\leq \varepsilon \norm[1]{w}.
  \end{displaymath}
  This proves the continuity of \(F(\uir{}{\infty}(1,b) g )\) at
  \(b=0\) and, by the group property, at any other point as well.
\end{proof}
\begin{rem}
  A direct evaluation shows, that the constant \(c=l(f_1)\) from the
  proof of Prop.~\ref{pr:nt-limit} for fiducial
  functional~\eqref{eq:fiducial-from-wavelet} is equal to
  \(c=\int_{\Space{R}{}} w(t)\,dt\). Of course, for non-trivial
  boundary values we need \(c\neq 0\). On the other hand, 
  admissibility condition~\eqref{eq:admissibility-weaker} requires
  \(c=0\). Moreover, the classical harmonic analysis and the
  traditional wavelet construction are two ``orthogonal'' parts of the
  same covariant transform theory in the following sense. We can
  present a rather general bounded function \(w=w_a+w_p\) as a sum of
  an admissible mother wavelet \(w_a\) and a suitable multiple \(w_p\)
  of the Poisson kernel. An extension of this technique to unbounded
  functions leads to \emph{Calder\'on--Zygmund decomposition}%
  \index{Calder\'on--Zygmund!decomposition}%
  \index{decomposition!Calder\'on--Zygmund}~\cite{Stein93}*{\S~I.4}.
\end{rem}

The table integral \(\int_{\Space{R}{}} \frac{dx}{x^2+1}=\pi\) tells
that the ``wavelet''
\(p(t)=\frac{1}{\pi}\frac{1}{1+t^2}\)~\eqref{eq:poisson-kernel} is in
\(\FSpace{L}{1}(\Space{R}{})\) with \(c=1\), the corresponding wavelet
transform is the Poisson integral. Its boundary behaviour from
Prop.~\ref{pr:nt-limit} is the classical result,
cf.~\cite{Garnett07a}*{Ch.~I, Cor.~3.2}.  The comparison of our
arguments with the traditional proofs, e.g.  in~\cite{Garnett07a},
does not reveal any significant distinctions. We simply made an
explicit usage of the relevant group structure, which is implicitly
employed in traditional texts anyway, cf.~\cite{Burenkov12a}. Further
demonstrations of this type can be found
in~\cites{Albargi13a,ElmabrokHutnik12a}.

\section{\large Transported norms}
\label{sec:transport-norm}

If the functional \(F\) and the representation \(\uir{}{}\)
in~\eqref{eq:coheret-transf-gen} are both linear, then the resulting
covariant transform is a linear map. If \(\oper{W}_F\) is injective,
e.g. due to irreducibility of \(\uir{}{}\),
then \(\oper{W}_F\) transports a norm \(\norm{\cdot}\) defined on
\(V\) to a norm \(\norm[F]{\cdot}\) defined on the image space
\(\oper{W}_F V\) by the simple rule:
\begin{equation}
  \label{eq:transported-norm}
  \norm[F]{u}:=\norm{v}, \quad
  \text{ where the unique }
  v\in V \text{ is defined by } u=\oper{W}_F v.
\end{equation}
By the very definition, we have the following
\begin{prop}
  \begin{enumerate}
  \item \(\oper{W}_F\) is an isometry \((V,\norm{\cdot})\rightarrow
    (\oper{W}_F V, \norm[F]{\cdot})\).
  \item If the representation \(\uir{}{}\) acts on
    \((V,\norm{\cdot})\) by isometries then \(\norm[F]{\cdot}\) is
    left invariant.
  \end{enumerate}
\end{prop}

A touch of non-triviality occurs if the transported norm can be
naturally expressed in the original terms of \(G\).
\begin{example}
  It is common to consider a unitary square integrable representation%
  \index{square integrable!representation}%
  \index{representation!square integrable} \(\uir{}{}\) and an
  admissible mother wavelet%
  \index{wavelet!mother!admissible}%
  \index{mother wavelet!admissible}%
  \index{admissible!mother wavelet} \(f\in V\).  In this case, 
  wavelet transform~\eqref{eq:wavelet-transf} becomes an isometry to
  square integrable functions on \(G\) with respect to a Haar
  measure~\cite{AliAntGaz00}*{Thm.~8.1.3}. In particular, for the
  affine group and setup of Example~\ref{ex:ax+b}, the wavelet
  transform with an admissible vector is a multiple of an isometry map
  from \(\FSpace{L}{2}(\Space{R}{})\) to the functions on the upper
  half-plane, i.e., the \(ax+b\) group, which are square integrable
  with respect to the Haar measure \(a^{-2}\,da\,db\).
\end{example}

A reader expects that there are other interesting examples of the
transported norms, which are not connected to the Haar integration.
\begin{example}
  In the setup of Example~\ref{ex:cauchy-integral}, consider the space
  \(\FSpace{L}{p}(\Space{R}{})\) with 
  representation~\eqref{eq:ax+b-repr-quasi-reg} of \(\mathrm{Aff}\) and 
  Poisson kernel \(p(t)\)~\eqref{eq:poisson-kernel} as an inadmissible
  mother wavelet. The norm transported by \(\oper{W}_P\) to the image
  space on \(\mathrm{Aff}\) is~\cite{Nikolski02a}*{\S~A.6.3}:
  \begin{equation}
    \label{eq:Hardy-norm-Aff}
    \norm[p]{u}=\sup_{a>0}\left(\int\limits_{-\infty}^\infty
      \modulus{u(a,b)}^p\,\frac{db}{a}\right)^{\frac{1}{p}}. 
  \end{equation}
  In the theory of Hardy spaces, the \(\FSpace{L}{p}\)-norm on the
  real line and transported norm~\eqref{eq:Hardy-norm-Aff} are
  naturally intertwined, cf.~\cite{Nikolski02a}*{Thm.~A.3.4.1(iii)}, and
  are used interchangeably.
\end{example}

The second possibility to transport a norm from \(V\) to a function
space on \(G\) uses an contravariant transform \(\oper{M}_v\):
\begin{equation}
  \label{eq:trans-norm-inverse}
  \norm[v]{u}:=\norm{\oper{M}_v u}.
\end{equation}
\begin{prop}
  \begin{enumerate}
  \item The contravariant transform \(\oper{M}_v\) is an isometry
    \((L,\norm[v]{\cdot})\rightarrow (V,\norm{\cdot})\).
  \item If the composition \(\oper{M}_v \circ \oper{W}_F=c I\) is a
    multiple of the identity on \(V\) then transported norms
    \(\norm[v]{\cdot}\)~\eqref{eq:trans-norm-inverse} and 
    \(\norm[F]{\cdot}\)~\eqref{eq:transported-norm} differ only by a
    constant multiplier.
  \end{enumerate}
\end{prop}
The above result is well-known for traditional wavelets.
\begin{example}
  In the setup of Example~\ref{ex:duflo-moor}, for a square integrable
  representation and two admissible mother wavelets \(v_0\) and
  \(w_0\) we know that \(\oper{M}_{w_0}\oper{W}_{v_0}=k
  I\)~\eqref{eq:wave-trans-inverse-ident}, thus transported norms
  \eqref{eq:transported-norm} and \eqref{eq:trans-norm-inverse} 
  differ by a constant multiplier. Thus, 
  norm~\eqref{eq:trans-norm-inverse} is also provided by the
  integration with respect to the Haar measure on \(G\).
\end{example}
In the theory of Hardy spaces the result is also classical.
\begin{example}
  For the fiducial functional \(F\) with 
  property~\eqref{eq:zero-limit-fiducial} and the contravariant
  transform \(\oper{M}_{v^*_0}^H\)~\eqref{eq:nt-limit},
  Prop.~\ref{pr:nt-limit} implies \(\oper{M}_{v^*_0}^H\circ
  \oper{W}_F=c I\). Thus, the norm transported to \(\mathrm{Aff}\) by
  \(\oper{M}_{v^*_0}^H\) from \(\FSpace{L}{p}(\Space{R}{})\) up to
  factor coincides with~\eqref{eq:Hardy-norm-Aff}. In other words, the
  transition to the boundary limit on the Hardy space is an isometric
  operator. This is again a classical result of the harmonic analysis,
  cf.~\cite{Nikolski02a}*{Thm.~A.3.4.1(ii)}.
\end{example}

The co- and contravariant transforms can be used to transport norms in
the opposite direction: from a classical space of functions on \(G\)
to a representation space \(V\).
\begin{example}
  Let \(V\) be the space of \(\sigma\)-finite signed measures of a
  bounded variation on the upper half-plane. Let the \(ax+b\) group
  acts on \(V\) by the representation adjoint to
  \([\uir{}{1}(a,b)f](x,y)=a^{-1}f(\frac{x-b}{a},\frac{y}{a})\) on
  \(\FSpace{L}{2}(\Space[+]{R}{2})\),
  cf.~\eqref{eq:ax+b-left-regular}. If the mother wavelet \(v_0\) is
  the indicator function of the square \(\{ 0<x<1, 0<y<1\}\), then the
  covariant transform of a measure \(\mu\) is
  \(\tilde{\mu}(a,b)=a^{-1}\mu(Q_{a,b})\), where \(Q_{a,b}\) is the
  square \(\{b<x<b+a, 0<y<a\}\). If we request that
  \(\tilde{\mu}(a,b)\) is a bounded function on the affine group, then
  \(\mu\) is a Carleson measure%
  \index{Carleson!measure}%
  \index{measure!Carleson}~\cite{Garnett07a}*{\S~I.5}. A norm
  transported from \(\FSpace{L}{\infty}(\mathrm{Aff})\) to the
  appropriate subset of \(V\) becomes the Carleson norm%
  \index{Carleson!norm}%
  \index{norm!Carleson} of measures. Indicator function of a
  tent\index{tent} taken as a mother wavelet will lead to an
  equivalent definition.
\end{example}

It was already mentioned in Rem.~\ref{re:grand-maximal-funct} and
Example~\ref{ex:atomic-decomposition} that we may be interested to
mix several different covariant and contravariant transforms. This
motivate the following statement.
\begin{prop}
  \label{pr:isometry-relation}
  Let \((V,\norm{\cdot})\) be a normed space and \(\uir{}{}\) be a
  continuous representation of a topological locally compact group
  \(G\) on \(V\). Let two fiducial operators \(F_1\) and \(F_2\)
  define the respective covariant transforms \(\oper{W}_1\) and
  \(\oper{W}_2\) to the same image space \(W=\oper{W}_1 V= \oper{W}_2
  V\).  Assume, there exists an contravariant transform
  \(\oper{M}: W \rightarrow V\) such that \(\oper{M}\circ
  \oper{W}_1=c_1 I\) and \(\oper{M}\circ \oper{W}_2=c_2 I\). Define by
  \(\norm[\oper{M}]{\cdot}\) the norm on \(U\) transpordef from \(V\)
  by \(\oper{M}\). Then
  \begin{equation}
    \label{eq:unitarity-property}
    \norm[\oper{M}]{\oper{W}_1 v_1+ \oper{W}_2 v_2}=\norm{c_1 v_1 +c_2
      v_2}, \quad \text{ for any } v_1, v_2\in V.
  \end{equation}
\end{prop}
\begin{proof}
  Indeed:
  \begin{equation*}
    \begin{split}
      \norm[\oper{M}]{\oper{W}_1 v_1+ \oper{W}_2 v_2}&=
      \norm{\oper{M}\circ\oper{W}_1 v_1+ \oper{M}\circ\oper{W}_2
        v_2}\\
      &=
      \norm{c_1 v_1+ c_2 v_2},
    \end{split}
  \end{equation*}
  by the definition of transported
  norm~\eqref{eq:trans-norm-inverse} and the assumptions
  \(\oper{M}\circ\oper{W}_i=c_i I\).
\end{proof}
Although the above result is simple, it does have important consequences.
\begin{cor}[Orthogonality Relation]
  Let \(\uir{}{}\) be a square integrable representation of a group
  \(G\) in a Hilbert space \(V\). Then, for any two admissible mother
  wavelets \(f\) and \(f'\) there exists a constant \(c\) such that:
  \begin{equation}
    \label{eq:orthogonality-wavelets}
    \int_G \scalar{v}{\uir{}{}(g)f}\,\overline{\scalar{v'}{\uir{}{}(g)f'}}\,dg
    =c\,\scalar{v}{v'} \quad \text{ for any }v_1,v_2\in V.
  \end{equation} 
  Moreover, the constant \(c=c(f',f)\) is a sesquilinear form of
  vectors \(f'\) and \(f\).
\end{cor}
\begin{proof}
  We can derive~\eqref{eq:orthogonality-wavelets}
  from~\eqref{eq:unitarity-property} as follows. Let \(\oper{M}_f\) be
  the inverse wavelet transform~\eqref{eq:wavelet-reconstruction}
  defined by the admissible vector \(f\), then \(\oper{M}_f\circ
  \oper{W}_f=I\) on \(V\) providing the right scaling of \(f\).
  Furthermore, \(\oper{M}_f\circ \oper{W}_{f'}=\bar{c}I\)
  by~\eqref{eq:wave-trans-inverse-ident} for some complex constant
  \(c\). Thus, by~\eqref{eq:unitarity-property}:
  \begin{displaymath}
    \norm[\oper{M}]{\oper{W}_f v +\oper{W}_{f'}v'}=\norm{v+\bar{c}v'}.
  \end{displaymath}
  Now, through the polarisation identity~\cite{KirGvi82}*{Problem~476}
  we get the equality~\eqref{eq:orthogonality-wavelets} of inner products.
\end{proof}
The above result is known as the \emph{orthogonality relation}%
\index{orthogonality!relation}%
\index{relation!orthogonality} in the theory of wavelets, for some
further properties of the constant \(c\)
see~\cite{AliAntGaz00}*{Thm.~8.2.1}.

Here is an application of Prop.~\ref{pr:isometry-relation} to harmonic
analysis, cf.~\cite{Grafakos08}*{Thm.~4.1.7}:
\begin{cor}
  \label{co:analytic-extension}
  The covariant transform \(\oper{W}_q\) with conjugate Poisson
  kernel \(q\)~\eqref{eq:conj-poisson-kernel} is a
  bounded map from \((\FSpace{L}{2}(\Space{R}{}),\norm{\cdot})\) to
  \((\FSpace{L}{}(\mathrm{Aff}), \norm[2]{\cdot})\) with norm
  \(\norm[2]{\cdot}\)~\eqref{eq:Hardy-norm-Aff}.  Moreover:
  \begin{displaymath}
    \norm[2]{\oper{W}_q f}= \norm{f},\qquad \text{ for all } f \in \FSpace{L}{2}(\Space{R}{}).
  \end{displaymath}
\end{cor}
\begin{proof}
  As we establish in Example~\ref{ex:hilbert-trans-analytic} for 
  contravariant transform
  \(\oper{M}_{{v}^+_0}^H\)~\eqref{eq:norm-limit},
  \(\oper{M}_{{v}^+_0}^H \circ \oper{W}_q=-\rmi I\) and \(\rmi I\) on
  \(\FSpace{H}{2}\) and \(\FSpace[\perp]{H}{2}\), respectively. 
  Take the unique presentation \(f=u+u^\perp\), for \(u\in
  \FSpace{H}{2}\) and \(u^\perp\in \FSpace[\perp]{H}{2}\). Then,
  by~\eqref{eq:unitarity-property}
  \begin{displaymath}
    \norm[2]{\oper{W}_q f}=\norm{-\rmi u+\rmi u ^\perp}= \norm{u+u^\perp}=\norm{f}.
  \end{displaymath}
  This completes the proof.
\end{proof}

\section{\large Conclusion}
\label{sec:conclusions}

We demonstrated that both, real and complex, techniques in harmonic
analysis have the same group-theoretical origin. Moreover, they are
complemented by the wavelet construction. Therefore, there is no any
confrontation between these approaches and they can be lined up as in
Table~\ref{tab:corr-betw-diff}.  In other words, the binary opposition
of the real and complex methods resolves via Kant's triad
thesis-antithesis-synthesis: complex-real-covariant.\medskip

\section*{\large Acknowledgements}

I am grateful to A. Albargi for careful
reading of the paper and useful comments. Prof.~D.~Yakubovich pointed
out some recent publications in the field. The anonymous referee made
many useful suggestions which were incorporated into the paper with
gratitude.  

\begin{sidewaystable}
  \begin{tabular}{||p{.31\textheight}|p{.31\textheight}|p{.31\textheight}||}
    \hline \hline
    \begin{center}
      \textbf{Covariant scheme} 
    \end{center}     & 
    \begin{center}
      \textbf{Complex variable}
    \end{center} &
    \begin{center}
      \textbf{Real variable}
    \end{center}\\
    \hline \hline Covariant transform%
    \index{covariant!transform}%
    \index{transform!covariant} is\par \(\oper{W}_F^{\uir{}{}}:
    v\mapsto \hat{v}(g) = F(\uir{}{}(g^{-1}) v)\).\par
    In particular, the wavelet transform%
    \index{wavelet!transform}%
    \index{transform!wavelet} for the mother wavelet \(v_0\) is
    \(\tilde {v}(g) = \scalar{ v}{\uir{}{}(g)v_0}\).  & The Cauchy
    integral%
    \index{integral!Cauchy}%
    \index{Cauchy!integral} is generated by the mother wavelet
    \(\frac{1}{2\pi\rmi}\frac{1}{x+\rmi}\). \par
    The Poisson integral%
    \index{Poisson kernel}%
    \index{kernel!Poisson} is generated by the mother wavelet
    \(\frac{1}{\pi}\frac{1}{x^2+1}\)& The averaging operator
    \(\tilde{f}(b)=\frac{1}{2a}\int\limits_{b-a}^{b+a} f(t)\,dt\) is
    defined by the mother wavelet \(\chi_{[-1,1]}(t)\),\par
    to average the modulus of \(f(t)\) we use all elements of the unit
    ball in
    \(\FSpace{L}{\infty}[-1,1]\).\\
    \hline The covariant transform maps vectors to functions on \(G\)
    or, in the induced case, to functions on the homogeneous space
    \(G/H\).& Functions are mapped from the real line to the upper
    half-plane parametrised by either the \(ax+b\)-group or the
    homogeneous space \(\SL/K\)\index{$\SL$}%
\index{group!$\SL$}.& Functions are mapped from the real
    line to the upper half-plane parametrised by either the
    \(ax+b\)-group or the homogeneous
    space \(\SL/A\).\\
    \hline Annihilating action on the mother wavelet produces
    functional relation on the image of the covariant transform & The
    operator \(-d\uir{\mathsf{A}}{} -\rmi
    d\uir{\mathsf{N}}{}=I+(x+\rmi)\frac{d}{dx}\) annihilates the mother
    wavelet \(\frac{1}{2\pi\rmi}\frac{1}{x+\rmi}\), thus the image of
    wavelet transform is in the kernel of the Cauchy--Riemann operator%
  \index{Cauchy--Riemann operator}%
  \index{operator!Cauchy--Riemann}
    \(-\linv{\mathsf{A}}+\rmi\linv{\mathsf{N}}=\rmi
    a(\partial_b+\rmi\partial_a)\). Similarly, for the Laplace
    operator.& The mother wavelet \(v_0=\chi_{[-1,1]}\) satisfies the equality
    \(\chi_{[-1,1]}=\chi_{[-1,0]}+\chi_{[0,1]}\), where both terms are
    again scaled and shifted \(v_0\). The image of the wavelet
    transform is suitable for the stopping time argument%
    \index{stopping time argument} and the dyadic squares%
    \index{dyadic!squares}%
    \index{squares!dyadic} technique.
    \\
    \hline An invariant pairing
    \(\scalar{\cdot}{\cdot}\) generates  the contravariant transform%
    \index{contravariant!transform}%
    \index{transform!contravariant}\par
    \([\oper{M}_{w_0}^{\uir{}{}} f]
    \scalar{f(g)}{\uir{}{}(g)w_0}\) for & The contravariant transform with the
    invariant Hardy pairing on the \(ax+b\) group produces boundary
    values of functions on the real line.&
    The covariant transform with the invariant \(\sup\) pairing
    produces the vertical%
  \index{vertical!maximal functions}%
  \index{maximal functions!vertical} and non-tangential%
  \index{non-tangential!maximal functions}%
  \index{maximal functions!non-tangential} maximal functions.
    \\
    \hline The composition \(\oper{M}_v \circ \oper{W}_F\) of the
    covariant and contravariant transforms is a multiple of the
    identity on  irreducible components.& SIO%
  \index{singular!integral operator}%
  \index{operator!integral!singular} is a composition of the Cauchy
  integral and its boundary value.  & The Hardy--Littlewood maximal function%
  \index{Hardy--Littlewood!maximal functions}%
  \index{maximal functions!Hardy--Littlewood} is the composition of the
  averaging operator and the contravariant transform from the invariant
  \(\sup\) pairing. 
    \\
    \hline The Hardy space is an invariant subspace of the group
    representation. & The Hardy space consists of the limiting values
    of the Cauchy integral. SIO is bounded on this space. & The
    Hardy--Littlewood maximal operator is bounded on
    the Hardy space \(\FSpace{H}{p}\) .
    \\
    \hline \hline
  \end{tabular}
  \medskip
  \caption{The correspondence between different elements of
    harmonic analysis.}
  \label{tab:corr-betw-diff}
\end{sidewaystable}

\newpage
{\small
\providecommand{\noopsort}[1]{} \providecommand{\printfirst}[2]{#1}
  \providecommand{\singleletter}[1]{#1} \providecommand{\switchargs}[2]{#2#1}
  \providecommand{\irm}{\textup{I}} \providecommand{\iirm}{\textup{II}}
  \providecommand{\vrm}{\textup{V}} \providecommand{\cprime}{'}
  \providecommand{\eprint}[2]{\texttt{#2}}
  \providecommand{\myeprint}[2]{\texttt{#2}}
  \providecommand{\arXiv}[1]{\myeprint{http://arXiv.org/abs/#1}{arXiv:#1}}
  \providecommand{\doi}[1]{\href{http://dx.doi.org/#1}{doi:
  #1}}\providecommand{\CPP}{\texttt{C++}}
  \providecommand{\NoWEB}{\texttt{noweb}}
  \providecommand{\MetaPost}{\texttt{Meta}\-\texttt{Post}}
  \providecommand{\GiNaC}{\textsf{GiNaC}}
  \providecommand{\pyGiNaC}{\textsf{pyGiNaC}}
  \providecommand{\Asymptote}{\texttt{Asymptote}}

}

\vskip 1 cm \footnotesize
\begin{flushleft}
Vladimir V. Kisil\\
School of Mathematics\\
University of Leeds\\
Leeds LS2\,9JT\\
UK\\
E-mail:
\href{mailto:kisilv@maths.leeds.ac.uk}{\texttt{kisilv@maths.leeds.ac.uk}}\\
Web: \href{http://www.maths.leeds.ac.uk/~kisilv/}%
{http://www.maths.leeds.ac.uk/\~{}kisilv/} 
\end{flushleft}

\vskip0.5cm
\begin{flushright}
Received: 12.09.2013
\end{flushright}

\printindex
\end{document}